\documentclass[leqno,final]{siamltex}
\usepackage{amsmath,amssymb}
\usepackage{graphicx}
\usepackage{epstopdf}
\usepackage{xcolor}
\usepackage{bm}
\allowdisplaybreaks[4]
\def \dis{\displaystyle}

\def \R{\mathbb{R}} 

\def \N{\mathbb{N}}

\def \N0{\mathbb{N}_0}
\def \a{\alpha}

\def \d{\delta}
\def \e{\varepsilon}

\def \W{\Omega}
\def \phi{\varphi}
\def \ge{\gamma_e}

\def \12{\dis\frac{1}{2}}

\def \1{\mathbbm{1}}

\def \<{\left<}
\def \>{\right>}

\def \L{\mathcal{L}}
\def \Th{\mathcal{T}_h}

\def \EI{\mathcal{E}_h^I}
\def \EB{\mathcal{E}_h^B}

\def \Wh{W_h^{2,p}}
\def \dv{\cdot\nu_{e}}
\def \grad{\nabla}
\def \lss{\lesssim}

\def \tI{{I}_h}
\def \aL{\L_h^*}

\newcommand{\Ome}{\Omega}
\newcommand{\mct}{\mathcal{T}_h}
\newcommand{\les}{\lesssim}
\newcommand{\nonum}{\nonumber}
\newcommand{\vh}{V_h}
\newcommand{\nab}{\nabla}
\newcommand{\mce}{\mathcal{E}_h}
\newcommand{\E}{\mathcal{E}_h}

\def \dx[#1]{\ensuremath{\operatorname{d}\!{#1}}}

\newtheorem{defn}{Definition}
\numberwithin{defn}{section}

\begin{document}

\title{Interior Penalty Discontinuous Galerkin Methods for Second Order Linear Non-Divergence Form Elliptic 
PDEs\thanks{The work of the first and third authors was partial supported by the NSF through grant DMS-1318486 
and the work of the second author was partially supported by the NSF grant DMS-1417980 and the Alfred Sloan 
Foundation.}}
\markboth{X. FENG, M. NEILAN and S. SCHNAKE}{IP-DG METHODS FOR LINEAR NON-DIVERGENCE FORM ELLIPTIC PDEs}

\author{Xiaobing Feng\thanks{Department of Mathematics, University of Tennessee, Knoxville, TN 37996 (xfeng@math.utk.edu).}
\and{Michael Neilan}\thanks{Department of Mathematics, University of Pittsburgh, Pittsburgh, PA 15260 (neilan@pitt.edu).} 
\and{Stefan Schnake}\thanks{Department of Mathematics, University of Tennessee, Knoxville, TN 37996 (schnake@math.utk.edu).}}

\date{}

\maketitle

\begin{abstract}
This paper develops interior penalty discontinuous Galerkin (IP-DG) methods 
to approximate $W^{2,p}$ strong solutions of second order linear elliptic partial differential equations (PDEs)
in non-divergence form with continuous coefficients. The proposed IP-DG methods are closely
related to the IP-DG methods for advection-diffusion equations, and they are easy to implement on 
existing standard IP-DG software platforms. It is proved that the proposed IP-DG methods have unique solutions
and converge with optimal rate to the $W^{2,p}$ strong solution in a discrete $W^{2,p}$-norm. The crux of the analysis 
is to establish a DG discrete counterpart of the Calderon-Zygmund estimate and to adapt a freezing 
coefficient technique used for the PDE analysis at the discrete level. As a 
byproduct of our analysis, we also establish broken $W^{1,p}$-norm error estimates for IP-DG approximations
of constant coefficient elliptic PDEs. Numerical experiments are 
 provided to gauge the performance of the proposed IP-DG methods and to validate the 
theoretical convergence results. 
\end{abstract}


\section{Introduction}\label{section1}

In this paper we develop interior penalty discontinuous Galerkin methods for approximating the $W^{2,p}$ strong
solution to the following second  order linear elliptic PDE in non-divergence form:
\begin{subequations}
\label{NonDivProblem}
\begin{alignat}{2}
\L u(x):= -A(x):D^2 u(x) &= f(x) &&\quad x\in \W, \label{eqn1:1}\\
u(x) &= 0 &&\quad x\in \partial\W, \label{eqn1:2}
\end{alignat}
\end{subequations}
where $\W\subset \R^n$ is an open, bounded domain with boundary $\partial\W$, $f\in L^p(\W)$ 
with $1< p<\infty$, and $A\in [C^0(\overline{\W})]^{n\times n}$ is positive definite 
in $\overline{\Omega}$.

Non-divergence form elliptic PDEs arrive naturally from many applications such as 
stochastic optimal control and game theory \cite{SP:Flem}; they are also encountered  in the 
linearization of fully nonlinear PDEs such as Monge-Amp\`ere-type equations \cite{AM:CG}.  
If $A$ is differentiable, then it is easy to check that equation \eqref{eqn1:1} can be rewritten as
a diffusion-convection equation with $A$ as the diffusion coefficient and ($\nabla \cdot A $) as the
convection coefficient. However, if $A\in \left[C^0(\overline{\W})\right]^{n\times n}$, 
then this formulation is not possible since  ($\nabla \cdot A$) does not exist as a function, but rather
only as a measure. 
We recall that in the literature there are three well-established PDE theories for non-divergence form elliptic 
PDEs 
depending on the smoothness of $A$ and $f$ (as well as the smoothness of the boundary
$\partial \Omega$ which we assume sufficiently smooth here to ease the presentation). The first theory 
is the classical solution (or Schauder's) theory \cite[Chapter 6]{Sp:DT} which seeks solutions 
in the H\"older space $C^{2,\alpha}(\overline{\Omega})$ for $0<\alpha<1$ when 
$A\in [C^\alpha(\overline{\W})]^{n\times n}$ and $f\in C^\alpha(\overline{\W})$. The second one 
is the $W^{2,p}$ (strong) solution theory  \cite[Chapter 9]{Sp:DT} which seeks solutions 
in the Sobolev space $W^{2,p}(\Omega)$ for $1<p<\infty$ that satisfy the PDE almost everywhere in 
$\Omega$ under the assumptions that $A\in [C^0(\overline{\W})]^{n\times n}$ and $f\in L^p(\W)$. 
If the coefficient matrix satisfies the Cordes condition and if the domain is convex,
then the notion of strong solutions may be extended to the case $A\in [L^\infty(\Omega)]^{n\times n}$ \cite{SmearsSuli,MaugeriBook}.
The third 
one is the viscosity (weak) solution theory \cite{Crandall_Ishii_Lions} which seeks solutions in $C^0(\W)$ 
(or even in $B(\Omega)$, the space of bounded functions in $\Omega$) that satisfy the PDE in the viscosity sense 
under the assumptions that $A\in [L^\infty(\Omega)]^{n\times n}$ and {$f\in L^\infty(\W)$}. 
We note that all the three solution concepts and PDE theories are non-variational; this is
a main difference between divergence form PDEs \cite[Chapter 8]{Sp:DT} and non-divergence form PDEs. 
We also note that in the case 
of the viscosity solution theory, the uniqueness and regularity of solutions had been the main focus of 
the study for second order non-divergence form PDEs (see \cite{Crandall_Ishii_Lions,SmearsSuli,Nadirashvili97,Safonov99} and 
the references therein).

In contrast to the advances of the PDE analysis, almost no progress on numerical methods and numerical 
analysis was achieved until very recently  
for second order elliptic PDEs in non-divergence form with non-differentiable coefficient matrix $A$ 
(cf. \cite{DednerPryer13,AX:FN,LakkisPryer11,Neilan13,NochettoZhang16,SmearsSuli,AX:WW}). 
The main difficulty is caused by the non-divergence structure of the PDEs which prevents any straightforward application 
of Galerkin-type numerical methodologies such as finite element methods, discontinuous Galerkin methods 
and spectral methods.  Moreover, the non-variational nature of the strong and viscosity solution concepts make 
convergence analysis and/or error estimates of any convergent numerical methods very delicate and difficult. Nonstandard
numerical techniques are often required to do the job (cf. \cite{AX:FN,Neilan13,NochettoZhang16,SmearsSuli,AX:WW}).

The primary goal of this paper is to develop convergent interior penalty discontinuous Galerkin (IP-DG) 
methods for approximating the $W^{2,p}$ strong solution of problem \eqref{NonDivProblem}
under the
assumption that $A\in \left[C^0(\overline{\W})\right]^{n\times n}$ and its solution satisfies 
the Calderon-Zygmund estimate \cite[Chapter 9]{Sp:DT}. This paper can be viewed as a DG counterpart of \cite{AX:FN} 
where non-standard convergent finite element methods were proposed and analyzed for approximating 
the $W^{2,p}$ strong solution of \eqref{NonDivProblem}.
The reason for undertaking such an extension is twofold. First, we intend to take advantage of 
some of the features of DG methods 
such as simplicity and ease of computation, and flexibility of mesh and ease for adaptivity, to design better 
numerical methods for problem \eqref{eqn1:1}--\eqref{eqn1:2}; in the meantime, we develop new numerical 
analysis techniques and machineries for non-divergence form PDEs. Second, since  the jumps of the normal 
derivatives across element edges and mesh-dependent bilinear forms are already used in the finite element 
methods of \cite{AX:FN}, it is natural to use totally discontinuous piecewise 
polynomial approximation spaces to explore the full potential of the interior penalty technique and idea. 
It should be noted that although the 
generalization of the finite element formulations of \cite{AX:FN} to the DG case is not very hard, 
the DG convergence analysis is more involved because of the extra difficulty caused by the 
non-conformity of the DG finite element space. The crux of 
the analysis of this paper is to establish a discrete Calderon-Zygmund estimate
for the proposed IP-DG methods and to adapt a freezing coefficient technique used in the PDE analysis
at the discrete level.  
{Moreover, in order to prove the desired discrete Calderon-Zygmund estimate, we need to
establish the $W^{1,p}_h$ stability and error estimates for the IP-DG approximations of constant coefficient
elliptic PDEs. Such estimates seem to be new and have independent interest.}  

The remainder of this paper is organized as follows. In Section \ref{section2} we provide the notation and a collection 
of preliminary estimates, {in particular, some properties of DG functions are either cited or proved. 
Section \ref{sec:NewConstant} analyzes the IP-DG methods for constant coefficient elliptic PDEs. 
The $W^{1,p}_h$ stability and error estimates are derived, which in turn lead to global $W^{2,p}_h$ 
stability estimates. 
The latter estimate can be regarded as a discrete Calderon-Zygmund estimate 
for the proposed IP-DG methods. To the best of our knowledge, such an estimate is new and of independent interest.} 
Section \ref{section4} is devoted to the formulation, stability and convergence analysis and error estimate
for the proposed IP-DG methods.  
{Here, using the continuity of the coefficient matrix $A$, the freezing coefficient technique is adapted 
to establish local stability 
(or a left-side inf-sup condition) for the IP-DG discrete operators, 
which together with a covering argument leads to a global G\"{a}rding-type inequality 
for the formal adjoint operators of the IP-DG operators. Next, using a duality argument we obtain a 
global left-side inf-sup condition for the IP-DG discrete operators. Finally, by employing another duality argument, we
derive the desired discrete Calderon-Zygmund estimate for the DG discrete operators.}
Once this stability estimate is shown, well-posedness and convergence of the IP-DG methods follow easily. 
In Section \ref{section5} we present a number of numerical experiments to verify our theoretical results 
and to gauge the performance of the proposed IP-DG methods, even for the case 
$A\in \left[L^\infty(\overline{\W})\right]^{n\times n}$ which is not covered by our convergence theory.

\section{Preliminary Results} \label{section2}

\subsection{Notation}

Let $\W$ be an open and bounded domain in $\R^n$.  For a subdomain $D$ of $\W$ with boundary $\partial D$,  let 
$L^p(D)$ and $W^{s,p}(D)$ for $s\geq 0$ and $1\leq p\leq\infty$ denote the standard Lebesgue and Sobolev spaces respectively,
and $W_0^{1,p}(D)$ be the closure of $C_c^{\infty}(D)$ in $W^{1,p}(D)$.  Let $(\cdot,\cdot)_D$ be the $L^2$ inner 
product on $D$ and $(\cdot,\cdot):=(\cdot,\cdot)_{\W}$.  To improve  the readability of the paper, we 
adopt the convention that $a\lss b$ stands for $a\leq Cb$ for some $C>0$ which does not depend on any discretization
parameters.

Let $\Th$ be a shape-regular and conforming triangulation of $\W$.  Let $\EI$ and $\EB$ denote respectively
the sets of all interior and boundary edges/faces of $\Th$, and set $\E:=\EI\cup \EB$. We introduce the broken Sobolev spaces
\begin{alignat*}{2}
W^{s,p}(\Th) &:= \prod_{T\in\Th}W^{s,p}(T), &&\qquad L^p(\Th):=W^{0,p}(\Th),\\
W^{s,p}_h(D) &:=W^{s,p}(\Th)\bigl|_D, &&\qquad L_h^p(D):= L^p(\Th)\bigr|_D.
\end{alignat*}
For any interior edge/face $e = \partial T^+ \cap \partial T^-\in \EI$ we define the jump and average of a scalar or 
vector valued function $v$ as
\begin{align*}
[v]\big|_{e} := v^+ - v^-, \qquad
\{v\}\big|_{e} := \frac{1}{2}\left(v^+ + v^-\right),
\end{align*}
where $v^{\pm}=v|_{T^\pm}$.  On a boundary edge/face $e\in \EB$
with $e = \partial T^+\cap \partial \Omega$, we set $[v]\big|_{e} =\{v\}\big|_e =  v^+$.
For any $e\in \EI$ we use $\nu_e$ to denote the unit outward normal vector pointing in the direction of the 
element with the smaller global index.  For $e\in \EB$ we set $\nu_e$ to be the outward normal to $\partial\Omega$ 
restricted to $e$.  The standard DG finite element space is defined as   
\begin{align*}
V_h &:= \Bigl\{v_h\in W^{2,p}(\Th);\, v_h\big|_T\in \mathbb{P}_k(T) \quad\forall T\in\Th \Bigr\},
\end{align*}
where $\mathbb{P}_k(T)$ denotes the set of all polynomials of degree less than or equal to $k$ on $T$.
We also introduce for any $D\subset\Omega$ 
\begin{align*}
V_h(D) := \Bigl\{v\in V_h;\, v\big|_{\Omega\setminus \overline{D}} \equiv 0 \Bigr\}.
\end{align*}
Note that $V_h(D)$ is nontrivial provided that there exists an inscribed ball $B$ with radius $r\geq h$ 
such that $B\subset D$.   {We also adopt the convention $V_h(\Omega) = V_h$.}

For each $e\in\E$, let {$\ge>0$} be constant on $e$.  We define the following mesh-dependent
norms on $W_h^{1,p}(D)$ and $W^{2,p}_h(D)$:
\begin{align}
\|v\|_{W^{2,p}_h(D)} &:= \|D^2_hv\|_{L^p(D)} + \Bigl(\sum_{e\in\EI}h_e^{1-p}
\big\| |[\grad v]| \big\|_{L^p(e\cap \bar{D})}^p\Bigr)^\frac{1}{p} \label{eqn2:2} \\
 &\qquad +\Bigl(\sum_{e\in\E}\gamma_e^ph_e^{1-2p}\|[v]\|_{L^p(e\cap \bar{D})}^p\Bigr)^\frac{1}{p}, \nonumber 
\\
\|v\|_{W_h^{1,p}(D)} &:= \|\grad_h v\|_{L^p(D)} 
+ \Bigl(\sum_{e\in\E}\ge^p h_e^{1-p}\|[v]\|_{L^p(e\cap \bar{D})}^p\Bigr)^\frac{1}{p} \\
&\qquad +\Bigl(\sum_{e\in \E} h_e \|\{\nab v\}\|_{L^p(e\cap \bar{D})}^p\Bigr)^\frac{1}{p}, \nonumber
\end{align}
where $\grad_h v$ and $D_h^2v$ denote the piecewise gradient and Hessian of $v$. 

In addition, we define the discrete $W^{-2,p}_h$-norm and $W_h^{-1,p}$-norm as follows:
\begin{align}
\|q\|_{W_h^{-2,p}(D)} &:= \sup_{0\neq v_h\in V_h}\frac{(q,v_h)_D}{\|v_h\|_{W_h^{2,p^\prime}(D)}},\label{eqn3:31} \\
\|q\|_{W_h^{-1,p}(D)} &:= \sup_{0\neq v\in W_h^{1,p'}(D)}\frac{(q,v)_D}{\|v\|_{W_h^{1,p^\prime}(D)}},
\end{align}
where $\frac{1}{p} +\frac{1}{p^\prime}=1$.
Finally, for any domain $D\subseteq\Omega$ and any $w\in L^p_h(D)$, we introduce the following mesh-dependent semi-norm
\begin{align} \label{L2norm}
\|w\|_{L^p_h(D)}:= \sup_{0\neq v_h\in V_h(D)} \frac{\bigl(w, v_h \bigr)_D}{\|v_h\|_{L^{p^\prime}(D)}}.
\end{align}
It can be proved that (cf. \cite{AX:FN})
\begin{align}\label{norm_equiv}
\|w_h\|_{L^p(\Omega)} \lss \|w_h\|_{L^p_h(\Omega)}\qquad \forall w_h\in V_h.
\end{align}

\subsection{Properties of the DG space $V_h$}
In this subsection we collect some technical lemmas that cover the basic properties of functions in the DG space $V_h$.  
These facts will be used many times in the later sections.  We first state the standard trace inequalities for 
broken Sobolev functions, a proof of this lemma can be found in \cite{Sp:BS}.
\begin{lemma}\label{lem2:2} 
For any $T\in\Th$ and $1<p<\infty$, there holds
\begin{align}
\|v\|_{L^p(\partial T)}^p\lss \Bigl(h_T^{p-1}\|\grad v\|_{L^p(T)}^p+h_T^{-1}\|v\|_{L^p(T)}^p\Bigr)
\qquad\forall v\in W^{1,p}(T).  \label{eqn2:3}
\end{align}
Therefore by scaling we have
\begin{align} \label{eqn2:12}
\sum_{e\in\EI}h_e\|v\|_{L^p(e\cap \overline{D})}^p \lss \begin{cases} \|v\|_{L^p(D)}^p 
&\forall v\in V_h(D), \\ \|v\|_{L^p(D)}^p +h^p\|\grad v\|_{L^p(D)}^p 
&\forall v\in W_h^{2,p}(D).
\end{cases}
\end{align}
\end{lemma}

Next, we prove an inverse inequality between the $\Wh$-norm and the $W_h^{1,p}$-norm.
\begin{lemma} \label{lem2:1}
For any $v_h\in V_h$, $D\subseteq\W$, there holds for $1 < p <\infty$
\begin{align}
\|v_h\|_{W^{2,p}_h(D)} &\lss h^{-1}\|v_h\|_{W_h^{1,p}(D_h)}, \label{eqn2:5}
\end{align}
where
\begin{align}
D_h = \{x\in\W;\,\operatorname{dist}(x,D)\leq h\}.
\end{align}
\end{lemma}

\begin{proof}
To show \eqref{eqn2:5}, we use \eqref{eqn2:2}, \eqref{eqn2:3}, and standard inverse estimates \cite{Sp:BS}
to obtain
\begin{align*}
\|v_h\|_{\Wh(D)} &\lss \|D_h^2v_h\|_{L^p(D)}+ \Bigl(\sum_{e\in\E}\ge^p h_e^{1-2p}\|[v]\|_{L^p(e\cap \bar{D})}^p\Bigr)^\frac{1}{p}\\
&\qquad
+\sum_{\substack{T\in\Th \\ T\subset D_h}}\Bigl(h_T^{1-p} \Bigl(h_T^{p-1}\|D^2v_h\|_{L^p(T)}^p+h_T^{-1}\|\grad v_h\|_{L^p(T)}^p
\Bigr)\Bigr)^\frac{1}{p} \\
&\lss\|D_h^2v_h\|_{L^p(D)}+ \Bigl(\sum_{e\in\E}\ge^p h_e^{-p}h_e^{1-p}\|[v]\|_{L^p(e\cap \bar{D})}^p\Bigr)^\frac{1}{p} \\ 
&\qquad
+\sum_{\substack{T\in\Th \\ T\subset D_h}}\Bigl(\|D^2v_h\|_{L^p(T)}^p+h_T^{-p}\|\grad v_h\|_{L^p(T)}^p\Bigr)^\frac{1}{p} \\
&\lss h^{-1}\|v_h\|_{W_h^{1,p}(D_h)} +  h^{-1}\Bigl(\sum_{e\in\EI}\ge^p h_e^{1-p}\|[v]\|_{L^p(e\cap \bar{D})}^p\Bigr)^\frac{1}{p}\\
&\lss h^{-1}\|v_h\|_{W_h^{1,p}(D_h)}.
\end{align*}\hfill
%
\end{proof}

%

We also prove an inverse inequality between the $L^p$-norm and the $W_h^{-1,p}$-norm. 

\begin{lemma}\label{lem2:5.5}
Let $v\in V_h(D)$.  For any $1< p<\infty$ and subdomain $D\subset\W$ we have
\begin{align}\label{eqn2:75}
\|v_h\|_{L^p(D)}\lss h^{-1}\|v_h\|_{W_h^{-1,p}(D)}.
\end{align}
\end{lemma}
\begin{proof}
Using the relation \eqref{norm_equiv} and the definition of $\|\cdot\|_{L^p_h(\Omega)}$, we find that
\begin{align*}
\|v_h\|_{L^p(D)}\le \|v_h\|_{L^p(\Omega)}\lss \|v_h\|_{L^p_h(\Omega)}
= \sup_{0 \neq w_h\in V_h} \frac{(v_h,w_h)_D}{\|w_h\|_{L^{p^\prime}(\Omega)}}\quad
\forall v_h\in V_h(D).
\end{align*}
Therefore, by the standard inverse
estimate $h \|w_h\|_{W_h^{1,p^\prime}(D)}\le h \|w_h\|_{W_h^{1,p^\prime}(\Omega)}\lss 
\|w_h\|_{L^{p^\prime}(\Omega)}$
and noting that $V_h\big(D)\subset W^{1,p^\prime}_h(D)$, we obtain
\begin{align*}
\|v_h\|_{L^p(D)}
&\lss h^{-1} \sup_{0\neq w_h\in V_h} \frac{(v_h,w_h)_D}{\|w_h\|_{W_h^{1,p^\prime}(D)}}
\le  h^{-1} \sup_{0\neq w\in W^{1,p^\prime}_h(D)} \frac{(v_h,w)_D}{\|w\|_{W_h^{1,p^\prime}(D)}}
=\|v_h\|_{W_h^{-1,p}(D)}.
\end{align*}
The proof is complete.
\end{proof}

The following lemma shows that the broken Sobolev norms are controlled by their  
corresponding Sobolev norms. 
%
\begin{lemma} \label{lem2:3.3}
For any $1<p<\infty$ there holds the following inequality:
\begin{align*}
\|\phi\|_{\Wh(\W)} &\leq \|\phi\|_{W^{2,p}(\W)} \quad\forall\phi\in W^{2,p}(\W)\cap W_0^{1,p}(\Omega).
\end{align*}
\end{lemma}
\begin{proof}
Since the inequality holds for all $\phi\in C^{\infty}(\W)\cap W^{1,p}_0(\Omega)$ and it can be extended to 
all $\phi\in W^{2,p}(\W)\cap W^{1,p}_0(\Omega)$ by a density argument.  
\end{proof}

The next lemma establishes a Poincar\'{e}-Friedrichs' inequality for DG functions.

\begin{lemma}\label{lem2:55}
Let $D\subset \Omega$ such that $V_h(D)\neq \{0\}$ and ${\rm diam}(D)\geq h$.
Then for any $v_h\in V_h(D)$ there holds the following inequalities:
\begin{align}
\|v_h\|_{L^{p}(D)} &\lss {\rm diam}(D)\|v_h\|_{W_h^{1,p}(D)} \label{eqn2.6:1},\\
\|v_h\|_{W_h^{1,p}(D)} &\lss {\rm diam}(D)  \|v_h\|_{\Wh(D)} \label{eqn2.6:2}.
\end{align}
\end{lemma}
\begin{proof}
Let $\tilde{V}_h$ denote the generalized Hsiegh--Clough--Tochner space \cite{DouglasDupont79},
and let
$E_h:V_h\to \tilde{V}_h$ 
be the reconstruction operator constructed in \cite{Houston11}.  The arguments given in \cite{Houston11}  
show that, for $v_h\in V_h(D)$,
\begin{align}\label{EhStability}
&E_h v_h\in H^2_0(D_h),\\
&\nonumber \|v_h - E_hv_h\|_{L^p(\Omega)} \lss h \|v_h\|_{W^{1,p}_h(D)},\\
&\nonumber \|v_h - E_hv_h\|_{W^{m,p}_h(\Omega)}\lss h^{s-m} \|v_h\|_{W^{s,p}_h(D)},\qquad 1\le m\le s\le 2,
\end{align}
where $D_h$ is the same as in Lemma \ref{lem2:1}.
Therefore, by the triangle inequality, the Poincar\`e-Friedrichs inequality, 
and the assumption ${\rm diam}(D)\geq h$,
\begin{align*}
\|v_h\|_{L^p(D)}&\le \|E_h v_h\|_{L^p(D)} +\|v_h- E_h v_h\|_{L^p(\Omega)}\\
&\lss {\rm diam}(D) \|E_h v_h\|_{W^{1,p}(D_h)} + h \|v_h\|_{W^{1,p}_h(D)} \lss  {\rm diam}(D)\|v_h\|_{W^{1,p}_h(D)}.
\end{align*}

Likewise, we find
\begin{align*}
\|v_h\|_{W^{1,p}_h(D)} 
&\le \|E_h v_h\|_{W^{1,p}(D_{h})}+ \|v_h- E_h v_h\|_{W^{1,p}_h(\Omega)},\\
&\lss {\rm diam}(D)\|E_hv_h\|_{W^{2,p}(D_{h})} + h\|v_h\|_{W^{2,p}_h(\Omega)}\lss {\rm diam}(D) \|v_h\|_{W^{2,p}_h(D)}.
\end{align*}
The proof is complete.
\end{proof}

Next we establish a discrete Sobolev interpolation estimate for DG functions.
%
\begin{lemma} \label{lem2:5}
Let $1< p< \infty$.  For all $v_h\in V_h$ we have
\begin{align}\label{eqn2:79}
\|v_h\|_{W_h^{1,p}(\W)}^2 \lss  \|v_h\|_{L^p(\W)} \|v_h\|_{\Wh(\W)}.
\end{align}
\end{lemma}
\begin{proof}
Let $E_h:V_h\to \tilde{V}_h$ be the enriching operator
in the proof of Lemma \ref{lem2:55}.
By the triangle inequality and scaling we find
\begin{align}\label{Line1ABC}
\|v_h\|_{W^{1,p}_h(\Omega)}^2 \les \|v_h-E_h v_h\|_{W^{1,p}_h(\Omega)}^2 + \|E_h v_h\|_{W^{1,p}(\Omega)}^2.
\end{align}
Since $E_h v_h\in W^{2,p}(\Omega)$ we can apply the  Gagliardo--Nirenberg estimate \cite{BrezisBook} to get
\begin{align*}
\|E_h v_h\|_{W^{1,p}(\Omega)}^2\les \|E_hv_h\|_{W^{2,p}(\Omega)} \|E v_h\|_{L^p(\Omega)}.
\end{align*}
Applying  estimates \eqref{EhStability}, we conclude that
\begin{align}\label{Line2ABC}
\| E_h v_h\|_{W^{1,p}(\Omega)}^2\les \|v_h\|_{W^{2,p}_h(\Omega)} \|v_h\|_{L^p(\Omega)}.
\end{align}
Likewise, by \eqref{EhStability} and an inverse estimate,
\begin{align}\label{Line3ABC}
\|v_h-E_h v_h\|_{W^{1,p}_h(\Omega)}^2 \les h^2 \|v_h \|_{W^{2,p}_h(\Omega)}^2\les \|v_h\|_{L^p(\Omega)} \|v_h\|_{W^{2,p}_h(\Omega)}.
\end{align}
Combining \eqref{Line1ABC}--\eqref{Line3ABC} completes the proof.
\end{proof}

Next we prove some local super approximation estimates for the DG nodal interpolation in various discrete norms.  
The derivation of the lemma is standard (cf. \cite{NitscheSchatz74});
{for completeness we provide the proof in the Appendix}.
\begin{lemma} \label{lem2:4}
Let ${I}_h:C^0(\Th):=\Pi_{T\in \Th} C^0(\overline{T}) \to V_h$ denote the nodal interpolation operator, 
and $\eta\in C^{\infty}(\W)$ with $|\eta|_{W^{j,\infty}(\W)}\lss d^{-j}$ for $0\leq j\leq k$. Then for 
any $v_h\in V_h$ and  $D\subseteq\W$ we have
\begin{align}
\|\eta v_h-\tI(\eta v_h)\|_{L^p(D)}&\lss \frac{h}{d}\|v_h\|_{L^p(D_h)}, \label{eqn2:37}\\
h\|\grad_h(\eta v_h-\tI(\eta v_h))\|_{L^p(D)}&\lss \frac{h}{d}\|v_h\|_{L^p(D_h)}, \label{eqn2:38}\\
h^2\|D_h^2(\eta v_h-\tI(\eta v_h))\|_{L^p(D)}&\lss \frac{h}{d}\|v_h\|_{L^p(D_h)}, \label{eqn2:38.5}\\
\|\eta v_h-\tI(\eta v_h)\|_{\Wh(D)}&\lss \frac{1}{d^2}\left(\|v_h\|_{L^p(D_h)}
+\|\grad_h v_h\|_{L^p(D_h)}\right),\label{eqn2:39}
\end{align}
where $D_h$ is the same as in Lemma \ref{lem2:1}.  Moreover, there holds
\begin{align} \label{eqn2:43}
\|\eta v_h-\tI(\eta v_h)\|_{\Wh(D)} \lss \dfrac{h}{d^3}\|v_h\|_{\Wh(D_h)} 
\end{align}
if the polynomial degree $k\geq 2$.  
\end{lemma}

\section{DG discrete $W^{1,p}$ and Calderon-Zygmund estimates for PDEs with constant coefficients}\label{sec:NewConstant}
In this section we consider the constant coefficient case, that is, $A(x)\equiv A_0\in \mathbb{R}^{n\times n}$ on $\W$.  
We { define three} interior-penalty discontinuous Galerkin 
discretizations $\L_{0,h}^{\e}$ to the PDE operator $\L$ and extend their domains to the broken Sobolev 
space $W^{2,p}(\Th)$. Our goal in this subsection is to prove global  
stability estimates for $\L_{0,h}^{\e}$ which will be crucially used in the next section. The 
final global stability estimate given in Theorem \ref{lem2:7} can be regarded as a DG 
discrete Calderon-Zygmund estimate for $\L_{0,h}^{\e}$.

Let $A_0$ be a constant, positive-definite matrix in $\R^{n\times n}$ and define
\begin{align}
\L_0 w:=-A_0:D^2w=-\grad\cdot(A_0\grad w).
\end{align}
From this we gather the standard PDE weak form:
\begin{align}
a_0(w,v):=\int_{\Omega}A_0\grad w\cdot\grad v\dx[x] \qquad\forall w,v\in H_0^1(\W).
\end{align}
The Lax-Milgram theorem \cite{Sp:BS} yields the existence and boundedness of $\L_0^{-1}:H^{-1}(\W)\to H_0^1(\W)$.  
Moreover if $\partial\W\in C^{1,1}$ we have from Calderon-Zygmund theory \cite{Sp:DT} that 
$\L_0^{-1}:L^p(\W)\to W^{2,p}(\W)\cap W_0^{1,p}(\W)$ exists and
\begin{align*}
	\|\L_{0}^{-1}\phi\|_{W^{2,p}(\W)}\lss \|\phi\|_{L^p(\W)} \qquad\forall \phi\in L^p(\W),
\end{align*}
and therefore
\begin{align*}
	\|w\|_{W^{2,p}(\W)}\lss \|\L_0 w\|_{L^p(\W)} \qquad\forall w\in W^{2,p}(\W)\cap W_0^{1,p}(\W). \label{eqn2:9}
\end{align*}

Define $\L_{0,h}^{\e}:V_h\to V_h$ by 
\begin{align}
 \bigl(\L_{0,h}^{\e}w_h&,v_h\bigr) := a_{0,h}^{\e}(w_h,v_h) \qquad\forall v_h,w_h\in V_h,
\end{align}
where the IP-DG bilinear form is defined by
\begin{align} \label{eqn2:7}
a_{0,h}(w_h,v_h) := &\int_\Omega A_0\grad_h w_h \cdot\grad_h v_h \dx[x]- \sum_{e\in\E}\int_e \{ A_0\grad w_h\dv \} [v_h] \dx[S] \\
& -\e \sum_{e\in\E}\int_e \{A_0\grad v_h\dv\}[w_h]\dx[S]+\sum_{e\in\E}\int_e \frac{\ge}{h_e}[w_h][v_h] \dx[S],  \nonumber
\end{align}
and $\gamma_e>0$ is a penalization parameter.
The parameter choices $\e \in \{ 1,0, -1\}$ 
give respectively the SIP-DG, IIP- DG, and NIP-DG formulations. For the sake of clarity and readability 
we shall assume for the rest of the paper that $\e$ may be either $1$, $0$, or $-1$ unless otherwise stated. 

We recall the following well-known DG integration by parts formula:
\begin{align}
\int_{\W}\tau\cdot\grad_h v \dx[x]= -\int_{\W}&(\grad_h\cdot\tau)v\dx[x] 
+\sum_{e\in\EI}\int_e [\tau\dv]\{v\} \dx[S] +\sum_{e\in\E}\int_e\{\tau\dv\}[v]\dx[S],\label{eqn2:8}
\end{align}
which holds for any piecewise scalar-valued function $v$ and vector-valued function $\tau$.  
Applying \eqref{eqn2:8} to the first term on the right-hand side of \eqref{eqn2:7} yields 
\begin{align}
a_{0,h}^{\e}(w_h,v_h) =& -\int_{\W}(A_0:D_h^2w_h)v_h \dx[x] + \sum_{e\in\EI}\int_e [ A_0\grad w_h\dv] \{v_h\} \dx[S] 
\label{eqn5}\\ 
&-\e \sum_{e\in\E}\int_e \{A_0\grad v_h\dv\}[w_h]\dx[S]+\sum_{e\in\E}\int_e \frac{\ge}{h_e}[w_h][v_h] \dx[S]\nonumber 
\end{align}
for any $w_h,v_h\in V_h$. By H\"older's inequality, it is easy to check that the above new form of $a_{0,h}^{\e}(\cdot,\cdot)$ 
is also well-defined on $W^{2,p}(\Th)\times W^{2,p'}(\Th)$ with $\frac{1}{p}+\frac{1}{p'}=1$.
As a result, this new form enables us to extend the domain of $a_{0,h}^{\e}(\cdot,\cdot)$ to
$W^{2,p}(\Th)\times W^{2,p'}(\Th)$ {and $\L^\e_{0,h}: W^{2,p}(\Th)\to (W^{2,p}(\Th))^*$.} 

\subsection{DG discrete $W^{1,p}$ error estimates}
From the standard IP-DG theory \cite{F:BR}, there exists $\gamma^* = \gamma^*(\|A_0\|_{L^\infty(\Omega)},
\mathcal{T}_h)>0$ depending only on the shape regularity of the mesh and on $\|A_0\|_{L^\infty(\Omega)}$
such that $\L_{0,h}^{\e}$ is invertible on $V_h$ provided $\ge\geq\gamma^*$; in the non--symmetric case $\e=-1$, 
$\gamma^*$ can be any positive number.  Moreover, if { $w\in W^{2,2}(\mct)\cap H^1_0(\W)$} and $w_h\in V_h$ 
satisfy
\begin{align}
\label{ErrorEquation}
a^\e_{0,h}(w-w_h,v_h)=0\qquad \forall v_h\in V_h,
\end{align}
then the quasi-optimal error estimate
\begin{align}\label{W12ErrorEstimate}
\|w-w_h\|_{W^{1,2}_h(\Omega)}\lss \inf_{v_h\in V_h}  \|w-v_h\|_{W^{1,2}_h(\Omega)}
\end{align}
is satisfied.  The goal of this subsection is to generalize this result to general exponent  $p\in (1,\infty)$ 
for the SIP-DG method. In particular, we have

\begin{theorem}\label{W1pTheorem}
Suppose {$w\in W^{2,p}(\W)\cap W^{1,p}_0(\Omega)\ (1<p<\infty)$}
and $w_h\in V_h$ satisfy \eqref{ErrorEquation} with $\e=1$.  Then there holds
\begin{align}\label{W1pTheoremMainLine}
\|w-w_h\|_{W^{1,p}_h(\Omega)}\lss h |\log h|^{t} \|w\|_{W^{2,p}(\Omega)},
\end{align}
where $t=(p+1)/p$ if $k=1$ and $t=0$ {if $k\geq 2$.}
\end{theorem}

To prove Theorem \ref{W1pTheorem}
we introduce some notation given in \cite{ChenChen04} 
(also see \cite{Schatz98}).  For given $z\in \overline{\Omega}$, we define the weight function $\sigma_z$ as
\begin{align}\label{sigma}
\sigma_z(x) = \frac{h}{|x-z|+h}.
\end{align}
%

For $p\in [1,\infty)$, and $s\in \mathbb{R}$, we define the following weighted norms
\begin{align}
\nonumber
\|v\|_{L^p(D),z,s} 
&= \left(\int_D \big|\sigma_z^s(x) v(x)\big|^p\, \dx[x]\right)^{{1}/p},\\
\nonumber
\|v\|_{W^{1,p}(D),z,s} 
& = \|v\|_{L^p(D),z,s} + \|\grad_h v\|_{L^p(D),z,s},\\
\label{weightedhnorm}
\|v\|_{W^{1,p}_h(D),z,s}
& = \|v\|_{W^{1,p}(D),z,s} + \Big(\sum_{e\in \mathcal{E}_h} h_e^{1-p} 
\big\|\sigma^s_z [v]\big\|_{L^p(e\cap \overline{D})}^p\Big)^{1/p}\\
&\nonumber \qquad 
+\Big(\sum_{e\in \mathcal{E}_h} h_e \big\|\sigma_z^s \{\grad_h v\}\big\|_{L^p(e\cap \overline{D})}^p\Big)^{1/p}.
\end{align}
The weighted norms in the case $p=\infty$ are defined analogously.

The derivation of $W^{1,p}$ error estimates of DG approximations is based on 
the work \cite{ChenChen04}, where localized pointwise
estimates of DG approximations are obtained.  There
it was shown that if $w\in W^{2,\infty}(\Omega)$
and $w_h\in V_h$ satisfy \eqref{ErrorEquation} with $\e=1$, then
\begin{align}\label{GradPointwiseChen}
|\nab (w-w_h)(z)|\lss \inf_{v_h\in V_h } \|w-v_h\|_{W^{1,\infty}_h(\Omega),z,s}\quad 0\le s< k
\end{align}
for all $z\in \overline{\Omega}$.
Similar to pointwise estimates of finite element approximations (e.g., \cite{Schatz98,Sp:BS}),
the ingredients to prove \eqref{GradPointwiseChen} include duality arguments
and DG approximation estimates of regularized Green functions 
in a weighted (discrete) $W^{1,1}$-norm.  These results are rather technical
and involve dyadic decompositions of $\Omega$, local DG error estimates, and Green function estimates. 

Here, we follow a similar argument to derive $W^{1,p}$ estimates; 
the main difference being that we derive
DG approximation estimates of regularized Green functions
in a weighted (discrete) $W^{1,p'}$-norm with $1/p+1/p'=1$
(cf. Lemma \ref{ezhatboundlemma}).  Using these estimates
and applying similar arguments in \cite{Schatz98,ChenChen04} 
then yield the estimate
\begin{align*}
|\nab (w-w_h)(z)|^p\lss h^{-n} \inf_{v_h\in V_h} \|w-v_h\|^p_{W^{1,p}_h(\Omega),z,s}
\end{align*}
for certain values of $s$.  Integrating this expression with respect
to $z$ and applying Fubini's theorem (cf. Lemma \ref{FubiniLemma})
then yields $L^p$ estimates of the piecewise gradient error.

Unfortunately, the strategy just described does
not immediately give us estimates for the terms $h_e^{1-p}\|[w-w_h]\|^p_{L^p(e)}$
appearing in the $W^{1,p}_h$-norm.  To bypass this difficulty,
we first use the trace inequality
\begin{align*}
\sum_{e\in \mce} h_e^{1-p}\|[w-w_h]\|_{L^p(e)}^p \lss  \|\nab_h (w-w_h)\|^p_{L^p(\Omega)}+ h^{-p}\|w-w_h\|^p_{L^p(\Omega)},
\end{align*}
and then derive estimates for $h^{-p} \|w-w_h\|^p_{L^p(\Omega)}$.  
We note that the standard duality argument to derive $L^p$ estimates yields 
\begin{align*}
\|w-w_h\|_{L^p(\Omega)}\lss  h \|w-w_h\|_{W^{1,p}_h(\Omega)},
\end{align*}
which is of little benefit.  Rather, our strategy is to modify
the arguments given in \cite[Theorem 5.1]{ChenChen04}
and  estimate $|(w-w_h)(z)|$ in terms of $\inf_{v_h\in V_h} \|w-v_h\|_{W^{1,p}_h(\Omega),z,s}$
(cf. Lemma \ref{PointWiseLemma})
and then apply Fubini's theorem.
We note that  it is due to this term that the $|\log h|^t$ factor
appears in Theorem \ref{W1pTheorem}.


\begin{lemma}\label{FubiniLemma}
Let $p\in [2,\infty)$ and $v\in L^p(\Ome)$.  Let $z\in \Omega$
and $T_z\in \mct$ such that $z\in T_z$.
Then there holds
\begin{align}\label{Fubini1}
\int_\Ome \int_{T_z} |v(x)|^p \dx[x] \dx[z] \les h^n \|v\|^p_{L^p(\Ome)}.
\end{align}
Moreover for any $s>n/p$ and $w\in W^{2,p}(\mct)$, there holds
\begin{align}\label{Fubini2}
\int_\Ome \|v\|_{W^{1,p}_h(\Ome),z,s}^p \dx[z] &\les \frac{h^n}{ps-n} \big(\|\nab_h v\|^p_{L^{p}(\Ome)}
+h^{-p}\|v\|^p_{L^p(\Omega)}+h^p \|D^2_h v\|^p_{L^p(\Omega)}\big).
\end{align}
If $s=n/p$, then we have
\begin{align}\label{Fubini3}
\int_\Ome \|v\|_{W^{1,p}_h(\Ome),z,n/p}^p \dx[z] &\les |\log h| {h^n} \big(\|\nab_h v\|^p_{L^{p}(\Ome)}
+h^{-p}\|v\|^p_{L^p(\Omega)}+h^p \|D^2_h v\|^p_{L^p(\Omega)}\big).
\end{align}
\end{lemma}
\begin{proof}
(i)
Let $v\in L^p(\Omega)$ and extend $v$ to $\mathbb{R}^n$ by zero.  Denote
by $B_{h}(z)$ the ball of radius $h$ and center $z$.  Then by a change of variables
and interchanging integrals, we find
\begin{align*}
\int_\Omega \int_{T_z} |v(x)|^p \dx[x]\dx[z]
&\le \int_\Omega \int_{B_h(z)} |v(x)|^p\dx[x]\dx[z] \\
& = h^n \int_\Omega \int_{B_1(0)} |v(z+hy)|^p \dx[y] \dx[z]\\
& = h^n  \int_{B_1(0)} \int_\Omega |v(z+hy)|^p \dx[z] \dx[y]\\
&\lss h^n \int_{B_1(0)} \|v\|_{L^p(\Omega)}^p\dx[y] \lss h^n \|v\|_{L^p(\Omega)}^p. 
\end{align*}
This proves \eqref{Fubini1}.

(ii) To prove \eqref{Fubini2} we again extend $v$  to $\mathbb{R}^n$ by zero
and make a change of variables to obtain
\begin{align*}
\int_\Omega \|\sigma_z^s v\|_{L^p(\Omega)}^p\dx[z]
& = \int_\Omega \int_\Omega  \Big(\frac{h}{|x-z|+h}\Big)^{sp} |v(x)|^p \dx[x]\dx[z]\\
& \le h^n \int_\Omega \int_{\hat{\Omega}} \frac{h^{sp}}{(|h y|+h)^{sp}} |v(z+h y)|^p\dx[y]\dx[z]\\
& = h^n \int_{\hat{\Omega}} \Big(\int_\Omega  |v(z+hy)|^p\dx[z]\Big)\frac{1}{(| y|+1)^{sp}} \dx[y],
\end{align*}
where
$\hat{\Omega} = \{ 2 h^{-1} x:\ x\in \Omega\}$
is a dilation of $\Omega$.
Therefore, 
\begin{align}\label{Line1}
\int_\Omega \|\sigma_z^s v\|_{L^p(\Omega)}^p\, \dx[z]
&\lss h^n \|v\|_{L^p(\Omega)}^p \int_{\hat{\Omega}} \frac{1}{(|y|+1)^{sp}}\, \dx[y].
\end{align}
For $sp>n$, there holds
\begin{align*}
 \int_{\hat{\Omega}}\frac{1}{(|y|+1)^{sp}}\dx[y]\lss \int_0^\infty \frac{r^{n-1}}{(r+1)^{sp}}\dx[r] = (n-1)! \prod_{j=1}^n (sp-j)^{-1}
 \le \frac{(n-1)!}{sp-n}.
 \end{align*}
 Combining this identity with \eqref{Line1} yields the inequality
 \begin{align}\label{FubiniproofLine1}
\int_\Omega \|\sigma_z^s v\|_{L^p(\Omega)}^p \dx[z]\lss \frac{h^n}{sp-n}  \|v\|_{L^p(\Omega)}^p.
\end{align}

If $sp=n$, then we find by a direct calculation that
\begin{align*}
\int_{\hat{\Omega}} \frac{1}{(|y|+1)^n}\dx[y]
&\lss \int_0^{h^{-1}} \frac{r^{n-1}}{(r+1)^n}\dx[r]
 = -\sum_{j=1}^{n-1} \frac{1}{(h+1)^{n-j}} + \log (1+h^{-1})\lss |\log h|,
\end{align*}
and therefore by \eqref{Line1},
\begin{align}\label{FubiniproofLine2}
\int_\Omega \|\sigma_z^{p/n} v\|_{L^p(\Omega)}^p\dx[z]\lss |\log h| h^n \|v\|_{L^p(\Omega)}.
\end{align}

Next, by trace inequalities given in Lemma \ref{lem2:2}, we have
\begin{align*}
\sum_{e\in \mathcal{E}_h} h_e^{1-p} \|[\sigma_z^s v]\|_{L^p(e)}^p
&\lss h^{-p} \|\sigma_z^s v\|_{L^p(\Omega)}^p+ \|\nabla_h (\sigma_z^s v)\|_{L^p(\Omega)}^p\\
&\lss h^{-p} \|\sigma_z^s v\|_{L^p(\Omega)}^p+ \| v \nabla (\sigma_z^s)\|_{L^p(\Omega)}^p+ \| \sigma_z^s \nabla_h v\|_{L^p(\Omega)}^p.
\end{align*}
Noting that
\begin{align*}
|\nabla (\sigma_z^s)|\lss \frac{h^s}{(|x-z|+h)^{s+1}}  = \frac{\sigma_z^s}{|x-z|+h}\lss h^{-1} \sigma_z^s,
\end{align*}
we obtain
\begin{align}\label{FubiniproofLine3}
\sum_{e\in \mathcal{E}_h} h_e^{1-p} \|[\sigma_z^s v]\|_{L^p(e)}^p\lss h^{-p} \|\sigma_z^s v\|_{L^p(\Omega)}^p+ \| \sigma_z^s \nabla_h v\|_{L^p(\Omega)}^p.
\end{align}
Likewise we have
\begin{align}\label{FubiniproofLine4}
\sum_{e\in \mathcal{E}_h} h_e \|\sigma_z^s \{\nabla v\}\|_{L^p(e)}^p\lss \|\sigma_z^s \nabla_h v\|_{L^p(\Omega)}^p+ h^p \|\sigma_z^s D^2_h v\|_{L^p(\Omega)}.
\end{align}
Combining \eqref{FubiniproofLine3}--\eqref{FubiniproofLine4} yields
\begin{align}\label{FubiniproofLine5}
\|v\|_{W^{1,p}_h(\Omega),z,s}\lss \|\sigma_z^s \nabla_h v\|_{L^p(\Omega)}^p+ h^{-p} \|\sigma_z^s  v\|_{L^p(\Omega)}^p
+h^p\|\sigma_z^s D^2_h v\|_{L^p(\Omega)}^p.
\end{align}
Finally applying the identities \eqref{FubiniproofLine1}--\eqref{FubiniproofLine2}
to \eqref{FubiniproofLine5} yields the desired result \eqref{Fubini2}--\eqref{Fubini3}.  The proof is complete.
\end{proof}

\begin{lemma}\label{PointWiseLemma}
Let $w\in W^{2,p}(\mct)\ (2\le p \le \infty)$ 
and $w_h\in V_h$ satisfy \eqref{ErrorEquation} with $\e=1$.
Then  for any $0\le s\le k-1+n/p$ and $z\in \overline{\Omega}$,
\begin{align*}
|(w-w_h)(z)|\les h^{1-n/p} |\log h|^{\bar{s}(p)} \inf_{v_h\in V_h} \|w-v_h\|_{W^{1,p}_h(\Omega),z,s},
\end{align*}
where $\bar{s}(p) = 1$ if $k = s+1-n/p$
and $\bar{s}(p) = 0$ for $k>s+1-n/p$.
\end{lemma}
\begin{proof}
{\em Step 1:  Set-up.} By the triangle inequality, an inverse estimate, and H\"{o}lder's inequality we obtain
\begin{align*}
|(w-w_h)(z)|
&\le |(w-v_h)(z)|+\|v_h-w_h\|_{L^\infty(T_z)}\\
&\le |(w-v_h)(z)|+h^{-n/2}\|v_h-w_h\|_{L^2(T_z)}\\
&\le |(w-v_h)(z)|+h^{-n/2}\big(\|w-w_h\|_{L^2(T_z)}+\|w-v_h\|_{L^2(T_z)}\big)\\
&\le \|w-v_h\|_{L^{\infty}(T_z)}+h^{-n/2}\|w-w_h\|_{L^2(T_z)}.
\end{align*}
Therefore by standard approximation theory,
and since $\sigma_z\approx 1$ on $T_z$, we have
\begin{align*}
|(w-w_h)(z)|
&\le h^{1-n/p} \|w\|_{W_h^{1,p}(T_z)}+h^{-n/2}\|w-w_h\|_{L^2(T_z)}\\
&\le h^{1-n/p} \|w\|_{W^{1,p}_h(\Omega),z,s}+h^{-n/2}\|w-w_h\|_{L^2(T_z)}.
\end{align*}
Replacing $w$ and $w_h$ by $w-v_h$ and $w_h-v_h$, respectively, 
yields
\begin{align}\label{SetUp1}
|(w-w_h)(z)|
&\lss h^{1-n/p} \|w-v_h\|_{W^{1,p}_h(\Omega),z,s}+h^{-n/2}\|w-w_h\|_{L^2(T_z)}.
\end{align}

Next, define $\rho\in L^2(\Omega)$ by
\begin{align*}
\rho(x) & = \left\{
\begin{array}{ll}
\frac{h^{-n/2} (w-w_h)(x)}{\|w-w_h\|_{L^2(T_z)}} & \text{if } x\in T_z\\
0 & \text{otherwise},
\end{array}
\right.
\end{align*}
and let $g_z\in H^1_0(\Omega)$ be the regularized
Green's function satisfying 
\begin{align}\label{gzLEquation}
\mathcal{L}_0 g_z = \rho.
\end{align}
Setting $g_{z,h}$ to be the DG
approximation  of $g_z$, i.e.,
$a_{0,h}(v_h,g_z-g_{z,h})=0,\ \forall v_h\in V_h$,
and $e_z:=g_z-g_{z,h}$, we have
by Galerkin orthogonality and the continuity
of the bilinear form,
\begin{align*}
h^{-n/2} \|w-w_h\|_{L^2(T_z)} 
&= (\rho,w-w_h)  = a_{0,h}(w-w_h,g_z)\\
& = a_{0,h}(w-v_h,e_z)\lss \|w-v_h\|_{W^{1,p}_h(\Omega),s,z}\|e_z\|_{W^{1,p'}_h(\Omega),-s,z}.
\end{align*}
Consequently, by \eqref{SetUp1}, we have
\begin{align}\label{SetUp2}
|(w-w_h)(z)|\lss  \|w-v_h\|_{W^{1,p}_h(\Omega),z,s}\big(h^{1-n/p}+\|e_z\|_{W^{1,p'}_h(\Omega),-s,z}\big)\quad \forall v_h\in V_h.
\end{align}
Thus, the proof will be completed once it is shown that $\|e_z\|_{W^{1,p'}_h(\Omega),-s,z}\lss |\log h|^{\bar{s}(p)} h^{1-n/p}$.
This result is derived in the following steps.

{\em Step 2: Dyadic decomposition of $\Omega$.} 
To estimate $\|e_z\|_{W^{1,p'}_h(\Omega),-s,z}$ we require some more notation.
Without loss of generality, assume that ${\rm diam}(\Omega)=1$. Let $d_j= 2^{-j}$ and set
\begin{align*}
\Omega_j &= \{x\in \Omega:\ d_{j+1}<|z-x|<d_j\},\\
\Omega^\prime_j &= \{x\in \Omega:\ d_{j+2}<|z-x|<d_{j-1}\},\\
\Omega^{\prime\prime}_j &= \{x\in \Omega:\ d_{j+3}<|z-x|<d_{j-2}\}.
\end{align*}

Let $M>1$ be a real number to be determined later, and let
$J\approx |\log h|$ be an integer such that $Mh = 2^{-J}$.
We then write
\begin{align}\label{Dyadic1}
\|e_z\|_{W^{1,p'}_h(\Omega),z,-s} \lss \|e_z\|_{W^{1,p'}_h(B_{Mh}(z)),z,{-s}}
+ \sum_{j=0}^J \|e_z\|_{W^{1,p'}_h(\Omega_j),z,-s}.
\end{align}
Note that, by the definition of $\Omega_j$,
 the weighted norms, and H\"older's inequality that
 \begin{align*}
 \|e_z\|_{W^{1,p'}_h(\Omega_j),z,-s}&\lss d_j^{n/q+s} h^{-s} \|e_z\|_{W^{1,2}_h(\Omega_j)},\\
 \|e_z\|_{W^{1,p'}_h(B_{Mh}),z,-s}&\lss h^{n/q} \|e_z\|_{W^{1,2}_h(B_{Mh}(z))}\le h^{n/q} \|e_z\|_{W^{1,2}_h(\Omega)},
 \end{align*}
 where 
 \begin{align*}
q\in [2,\infty]\ \text{satisfies}\ 1/q+1/p = 1/2.
 \end{align*}
 Applying these estimates to \eqref{Dyadic1} 
yields
\begin{align}\label{Dyadic2}
\|e_z\|_{W^{1,p'}_h(\Omega),z,-s} 
&\lss h^{n/q} \|e_z\|_{W^{1,2}_h(\Omega)}
+ \sum_{j=0}^J d_j^{n/q+s} h^{-s} \|e_z\|_{W^{1,2}_h(\Omega_j)}\\
&\nonumber= h^{n/q} \|e_z\|_{W^{1,2}_h(\Omega)}+Q_h,
\end{align}
where
\begin{align}\label{QhDef}
Q_h:=h^{-s}\sum_{j=0}^J d_j^{n/q+s}  \|e_z\|_{W^{1,2}_h(\Omega_j)}.
\end{align}

To estimate the first term in the right--hand side of \eqref{Dyadic2},
we apply elliptic regularity and the identity $\|\rho\|_{L^2(\Omega)}= h^{-n/2}$
to obtain
\begin{align*}
\|e_z\|_{W^{1,2}_h(\Omega)}\lss h \|g_z\|_{W^{2,2}(\Omega)}\lss h \|\rho\|_{L^2(\Omega)} = h^{1-n/2}.
\end{align*}
Applying this estimate in \eqref{Dyadic2}
and using the identity $1-n/2+n/q = 1-n/p$
yields
\begin{align}\label{Dyadic3}
\|e_z\|_{W^{1,p'}_h(\Omega),z,-s}
\lss h^{1-n/p}+ Q_h.
\end{align}
It remains to find an appropriate upper bound of $Q_h$ to complete the proof.

{\em Step 3: Estimate of $Q_h$ --Local error estimates.} Lemma 4.4 in \cite{ChenChen04} states that 
\begin{align*}
\|e_z\|_{W^{1,2}_h(\Omega_j)}
%
&\lss h^k d_j^{1-k-n/2}+ d_j^{-1} \|e_z\|_{L^2(\Omega_j^\prime)}.
\end{align*}
Applying this estimate to the definition of $Q_h$ \eqref{QhDef} yields
\begin{align*}
Q_h
&\lss h^{k-s} \sum_{j=0}^J d_j^{n/q+s+1-k-n/2} +h^{-s}\sum_{j=0}^J d_j^{n/q+s-1} \|e_z\|_{L^2(\Omega_j^\prime)}\\
& = h^{k-s} \sum_{j=0}^J d_j^{-(k-s+n/p-1)} +h^{-s}\sum_{j=0}^J d_j^{n/q+s-1} \|e_z\|_{L^2(\Omega_j^\prime)}\\
& = h^{1-n/p} \Theta(k-s+n/p-1)
+ h^{-s} \sum_{j=0}^J d_j^{n/q+s-1} \|e_z\|_{L^2(\Omega_j^\prime)},
\end{align*}
where
\begin{align*}
\Theta(\gamma):=\sum_{j=0}^J \big(\frac{h}{d_j})^\gamma.
\end{align*}

Therefore, since (cf. \cite[(5.19)]{ChenChen04})
\begin{align*}
\Theta(\gamma)\lss 
\left\{
\begin{array}{ll}
|\log h| & \text{if }\gamma=0,\\
\frac{1}{M^\gamma (1-2^{-\gamma})} & \text{if }\gamma>0,
\end{array}
\right.
\end{align*}
we find that
\begin{align}\label{QhInequality1}
Q_h \lss |\log h|^{\bar{s}(p)} h^{1-n/p} + h^{-s} \sum_{j=0}^J d_j^{n/q+s-1}\|e_z\|_{L^2(\Omega_j^\prime)}.
\end{align}

{\em Step 4: Estimate of $Q_h$ -- Duality Arguments.} 
Applying \cite[(5.24)]{ChenChen04} yields
\begin{align}\label{DualityEz}
\|e_z\|_{L^2(\Omega^\prime_j)}
&\lss h^k d_j^{1-k-n/2} \|e_z\|_{W^{1,1}_h(\Omega)} + h\|e_z\|_{W^{1,2}_h(\Omega_j^{\prime\prime})}.
\end{align}
Using  estimates \eqref{DualityEz} and \eqref{QhInequality1},
and noting that $\max_{0\le j\le J} d_j^{-1} = 2^J = 1/(hM)$, we find
\begin{align*}
Q_h 
&\lss |\log h|^{\bar{s}(p)} h^{1-n/p}
+ h^{k-s} \sum_{j=0}^J d_j^{s-k-n/p} \|e_z\|_{W^{1,1}_h(\Omega)}
+ h^{1-s} \sum_{j=0}^J d_j^{n/q+s-1} \|e_z\|_{W^{1,2}_h(\Omega_j^{\prime\prime})}\\
&\lss |\log h|^{\bar{s}(p)} h^{1-n/p}
+ h^{k-s} \sum_{j=0}^J d_j^{s-k-n/p} \|e_z\|_{W^{1,1}_h(\Omega)}
+ \frac{h^{-s}}{M} \sum_{j=0}^J d_j^{n/q+s} \|e_z\|_{W^{1,2}_h(\Omega_j^{\prime\prime})}\\
&\lss |\log h|^{\bar{s}(p)} h^{1-n/p}
+ h^{-n/p}\Theta(k-s+n/p)\|e_z\|_{W^{1,1}_h(\Omega)}+ \frac{1}{M} Q_h.
\end{align*}
Taking $M$ sufficiently large yields
\begin{align*}
Q_h \lss  |\log h|^{\bar{s}(p)} h^{1-n/p}
+ h^{-n/p} \Theta(k-s+n/p)\|e_z\|_{W^{1,1}_h(\Omega)}.
\end{align*}

Applying this estimate to \eqref{Dyadic3} then yields
\begin{align}\label{Dyadic5}
\|e_z\|_{W^{1,p}_h(\Omega),z,-s}\lss |\log h|^{\bar{s}(p)} h^{1-n/p}
+ h^{-n/p}\Theta(k-s+n/p)\|e_z\|_{W^{1,1}_h(\Omega)}.
\end{align}
In particular, the case $s=0$, $p=\infty$, $p'=1$ gives
\begin{align*}
\|e_z\|_{W^{1,1}_h(\Omega)}\lss |\log h|^{\bar{s}(\infty)} h
+ \Theta(k)\|e_z\|_{W^{1,1}_h(\Omega)}.
\end{align*}
Since 
\begin{align*}
\Theta(k)\lss \frac{1}{M^k(1-2^{-k})},
\end{align*}
we can take $M$ sufficiently large to conclude that
\begin{align*}
\|e_z\|_{W^{1,1}_h(\Omega)}\lss |\log h|^{\bar{s}(\infty)} h.
\end{align*}
Finally, applying this last estimate to \eqref{Dyadic5} yields
\begin{align*}
\|e_z\|_{W^{1,p}_h(\Omega),z,-s}\lss |\log h|^{\bar{s}(p)} h^{1-n/p}\big(1
+ \Theta(k-s+n/p)\big)\lss  |\log h|^{\bar{s}(p)} h^{1-n/p}.
\end{align*}
Applying this last estimate to \eqref{SetUp2} completes the proof.
\end{proof}

\begin{lemma}\label{ezhatboundlemma}
Let $z$ and $T_z$ be as in Lemma \ref{FubiniLemma}.
For arbitrary $\varphi\in C^{\infty}_0(T_z)$, with $\|\varphi\|_{W^{1,2}(T_z)}=1$, 
we extend $\varphi$ to $\Ome$ by zero, and
let $\hat{g}_z$ be the solution to
\begin{subequations}
\label{ghat}
\begin{alignat*}{2}
\mathcal{L}^*_0\hat{g}_z=h^{-n/2-1} \partial \varphi/\partial x_i\quad  &&\text{in }\Ome,\qquad
\hat{g}_z  = 0\quad \text{on }\partial\Ome.
\end{alignat*}
\end{subequations}
Let $\hat{g}_{z,h}\in V_h$ satisfy the discrete adjoint problem
\begin{align*}
a_{0,h}(v_h,\hat{g}_{z,h}) =h^{-n/2-1}\int_\Omega (\partial \varphi/\partial x_i) v_h\, \dx[x]\qquad \forall v_h\in V_h,
\end{align*}
where we have dropped the superscript of the bilinear form for notational simplicity. 
 Let $p\in [2,\infty],\ p^\prime \in [1,2]$ such that $1/p+1/p^\prime = 1$.  Then 
 for any $0\le s\le k+n/p$ there holds
 \begin{align*}
 \|\hat{g}_z-\hat{g}_{z,h}\|_{W^{1,p^\prime}_h(\Ome),z,-s}&\les |\log h|^{\bar{\bar{s}}(p)} h^{-n/p},
 \end{align*}
 where $\bar{\bar{s}}(p) = 1$ if $s = k+n/p$ and $\bar{\bar{s}}(p)=0$ otherwise.
\end{lemma}
\begin{proof}
Set $\hat{e}_z  = \hat{g}_z-\hat{g}_{z,h}$, and for $M>0$, let $J$ satisfy
$Mh = 2^{-J}$.
Then by applying 
similar arguments as the proof of Lemma \ref{PointWiseLemma}, we obtain
\begin{align*}
\|\hat{e}_z\|_{W^{1,p'}_h(\Omega),z,-s}
&\le \|\hat{e}_z\|_{W^{1,p'}_h(B_{Mh}(z)),z,-s} + \sum_{j=0}^J \|\hat{e}_z\|_{W^{1,p'}_h(\Omega_j)}\\
&\lss h^{n/q} \|\hat{e}_z\|_{W^{1,2}_h(\Omega)} + h^{-s} \sum_{j=0}^J d_j^{n/q+s} \|\hat{e}_z\|_{W^{1,2}_h(\Omega_j)}\\
&\lss h^{n/q+1} \|\hat{g}_z\|_{W^{2,2}_h(\Omega)} + h^{-s} \sum_{j=0}^J d_j^{n/q+s} \|\hat{e}_z\|_{W^{1,2}_h(\Omega_j)}\\
&\lss h^{-n/p} + \hat{F}_h,
\end{align*}
with 
\begin{align}\label{FhatDef}
\hat{F}_h:=h^{-s} \sum_{j=0}^J d_j^{n/q+s} \|\hat{e}_z\|_{W^{1,2}_h(\Omega_j)}.
\end{align}
By the local error estimate given in \cite[Lemma 4.2]{ChenChen04} we have
\begin{align*}
\|\hat{e}_z\|_{W^{1,2}_h(\Omega_j)}\lss h^k \|\hat{g}_z\|_{W^{k+1,2}(\Omega_j^\prime)} + d_j^{-1}\|\hat{e}_z\|_{L^2(\Omega_j^\prime)},
\end{align*}
and Green function estimates show that $\|\hat{g}_z\|_{W^{k+1,2}(\Omega_j^\prime)}\lss d_j^{-n/2-k}$.
Applying these estimates into \eqref{FhatDef} yield
\begin{align*}
\hat{F}_h
&\lss h^{k-s} \sum_{j=0}^J d_j^{n(1/q-1/2)+s-k} + h^{-s} \sum_{j=0}^J d_j^{n/q+s-1} \|\hat{e}_z\|_{L^2(\Omega_j^\prime)}\\
& = h^{-n/p} \Theta(k-s+n/p) + h^{-s} \sum_{j=0}^J d_j^{n/q+s-1} \|\hat{e}_z\|_{L^2(\Omega_j^\prime)}\\
&\lss |\log h|^{\bar{\bar{s}}(p)} h^{-n/p}  + h^{-s} \sum_{j=0}^J d_j^{n/q+s-1} \|\hat{e}_z\|_{L^2(\Omega_j^\prime)}.
\end{align*}

Applying \cite[(5.39)]{ChenChen04}, we have
\begin{align*}
\|\hat{e}_z\|_{L^2(\Omega_j^\prime)}\lss h^k d_j^{1-k-n/2} \|\hat{e}_z\|_{W^{1,1}_h(\Omega)}+h \|\hat{e}_z\|_{W^{1,2}_h(\Omega_j^{\prime\prime})},
\end{align*}
and therefore
\begin{align*}
\hat{F}_h
&\lss |\log h|^{\bar{\bar{s}}(p)}h^{-n/p} + h^{-n/p}\Theta(k-s+n/p) \|\hat{e}_z\|_{W^{1,1}_h(\Omega)}+ h^{1-s} \sum_{j=0}^J d_j^{n/q+s-1} 
 \|\hat{e}_z\|_{W^{1,2}_h(\Omega_j^{\prime\prime})}\\
 &\lss |\log h|^{\bar{\bar{s}}(p)} h^{-n/p} + h^{-n/p}\Theta(k-s+n/p) \|\hat{e}_z\|_{W^{1,1}_h(\Omega)}+ \frac{\hat{F}_h}{M}.
  \end{align*}
  By taking $M$ sufficiently large, we obtain
\begin{align*}
\hat{F}_h\lss |\log h|^{\bar{\bar{s}}(p)} h^{-n/p} + h^{-n/p}\Theta(k-s+n/p) \|\hat{e}_z\|_{W^{1,1}_h(\Omega)},
\end{align*}
and therefore
\begin{align*}
\|\hat{e}_z\|_{W^{1,p'}_h(\Omega),z,-s}\lss |\log h|^{\bar{\bar{s}}(p)} h^{-n/p} + h^{-n/p}\Theta(k-s+n/p) \|\hat{e}_z\|_{W^{1,1}_h(\Omega)}.
\end{align*}
The case $s=0$, $p'=1$, $p=\infty$ yields
\begin{align*}
\|\hat{e}_z\|_{W^{1,1}_h(\Omega)}\lss  1 + \Theta(k) \|\hat{e}_z\|_{W^{1,1}_h(\Omega)},
\end{align*}
and therefore, we conclude by taking $M>0$ sufficiently large that
\begin{align*}
\|\hat{e}_z\|_{W^{1,1}_h(\Omega)}\lss  1
\end{align*}
We then conclude that
\begin{align*}
\|\hat{e}_z\|_{W^{1,p'}_h(\Omega),z,-s}\lss |\log h|^{\bar{\bar{s}}(p)} h^{-n/p}.
\end{align*}
The proof is complete.
\end{proof}

\subsubsection{Proof of Theorem \ref{W1pTheorem} for ${\bm  p}$ $\mathbf{\geq 2}$}
We now prove Theorem \ref{W1pTheorem} in 
the case $p\in [2,\infty)$.
To this end, let $z\in \Omega$ and $T_z\in \mct$ such that
$z\in T_z$.  Using an inverse estimate, \eqref{eqn2:75}, and the
  triangle inequality
we obtain 
\begin{align}\label{prelimLp}
|\partial w_h(z)/\partial x_i|
&\les h^{-n/2}\|\partial w_h/\partial x_i\|_{L^2(T_z)} \\
&\nonum \les h^{-n/2-1}\|\partial  w_h/\partial x_i\|_{W^{-1,2}(T_z)}\\
&\nonum\les h^{-n/2-1}\Bigl(\|\partial (w-w_h)/\partial x_i\|_{W^{-1,2}(T_z)} +\|\partial w/\partial x_i\|_{W^{-1,2}(T_z)}\Bigr).
\end{align}
Note that, by the {Poincar\'e-Friedrichs} and H\"older inequalities, 
\begin{align*}
\|\partial w/\partial x_i\|_{W^{-1,2}(T_z)} 
&= \mathop{\sup_{\varphi\in C^\infty_0(T_z)}}_{\|\varphi\|_{W^{1,2}(T_z)}=1} (\partial w/\partial x_i,\varphi)_{T_z}\\
&\les \mathop{\sup_{\varphi\in C^\infty_0(T_z)}}_{\|\varphi\|_{W^{1,2}(T_z)}=1} |T_z|^{\frac{p-2}{2p}} \|\partial w/\partial x_i\|_{L^p(T_z)} \|\varphi\|_{L^2(T_z)}\\
&\les |T_z|^{\frac{p-2}{2p}} {\rm diam}(T_z) \|\partial w/\partial x_i\|_{L^p(T_z)} \les h^{1+n/2-n/p}\|\partial w/\partial x_i\|_{L^p(T_z)}.
\end{align*}
Inserting this estimate into \eqref{prelimLp} yields
\begin{align*}
|\partial w_h(z)/\partial x_i|
&\nonum\les h^{-n/p} \|\partial w/\partial x_i\|_{L^p(T_z)}+h^{-n/2-1}\|\partial (w-w_h)/\partial x_i\|_{W^{-1,2}(T_z)}.
\end{align*}
Replacing $w$ by $w-v_h$ and $w_h$ by $w_h-v_h$ for some $v_h\in \vh$ in the argument above, we conclude
\begin{align}\label{prelimLp2}
|\partial (w_h-v_h)(z)/\partial x_i| 
&\les h^{-n/p} \|\partial (w-v_h)/\partial x_i\|_{L^p(T_z)}\\
&\qquad\nonumber +h^{-n/2-1}\|\partial (w-w_h)/\partial x_i\|_{W^{-1,2}(T_z)}.
\end{align}

Let $\varphi$, $\hat{g}_z$ and $\hat{g}_{z,h}$ be as in Lemma \ref{ezhatboundlemma}.
Setting $\hat{e}_z = \hat{g}_z-\hat{g}_{z,h}$, we have for arbitrary $v_h\in \vh$
\begin{align*}
h^{-n/2-1}\int_{T_z}(w-w_h)\partial \varphi/\partial x_i\, dx
& = a_{0,h}(w-v_h,\hat{e}_z)\\
&\les  \|w-v_h\|_{W_h^{1,p}(\Ome),z,s} \|\hat{e}_z\|_{W^{1,p^\prime}_h(\Ome),z,-s}\\
&\les \|w-v_h\|_{W_h^{1,p}(\Ome),z,s} |\log h|^{\bar{\bar{s}}(p)} h^{-n/p},
\end{align*}
where $\bar{\bar{s}}(p)$ is defined in Lemma \ref{ezhatboundlemma}.

Applying this last estimate into \eqref{prelimLp2} yields
 \begin{align}\label{powerP}
|\nab (w_h-v_h)(z)| 
&\les h^{-n/p} \|\nab (w-v_h)\|_{L^p(T_z)}\\
&\nonum\qquad\qquad +h^{-n/p} |\log h|^{\bar{\bar{s}}(p)}\|w-v_h\|_{W^{1,p}_h(\Ome),z,s}.
\end{align}
Raising \eqref{powerP} by the power $p$ and integrating over $\Ome$ with respect to $z$, we conclude
\begin{align*}
\|\nab_h (w_h-v_h)\|_{L^p(\Ome)}
&\les \left(h^{-n}\int_\Ome \|\nab_h (w-v_h)\|_{L^p(T_z)}^p\, \dx[z]\right)^{1/p}\\
&\qquad\nonum +\left(h^{-n} |\log h|^{\bar{\bar{s}}(p)p}\int_\Ome \|w-v_h\|^p_{W^{1,p}_h(\Ome),z,s}\, \dx[z]\right)^{1/p}.
\end{align*}
Next, we choose $s$ such that $n/p<s<k+n/p$.  Then ${\bar{\bar{s}}(p)}=0$, and 
by \eqref{Fubini1}--\eqref{Fubini2}
\begin{align*}
\|\nab_h (w_h-v_h)\|^p_{L^p(\Ome)} &\les \|\nab_h(w-v_h)\|^p_{L^{p}(\Ome)}+ h^{-p} \|w-v_h\|_{L^p(\Omega)}^p + h^p \|D^2_h(w-v_h)\|_{L^p(\Omega)}^p,
\end{align*}
and therefore by the triangle inequality, and by taking $v_h = I_h w$, the nodal interpolant of $w$,
\begin{align}\label{W1pEstimatePart1}
\|\nab_h (w-w_h)\|_{L^p(\Ome)}\les h \|w\|_{W^{2,p}(\Omega)}.
\end{align}
Next we {bound} the jumps $\|[w-w_h]\|_{L^p(e)}$.
First, by the trace inequalities stated in Lemma \ref{lem2:2} we have
\begin{align}\label{asdf0}
\sum_{e\in \mce} h_e^{1-p} \|[w-w_h]\|_{L^p(e)}^p
\lss C \Big(\|\nab_h (w-w_h)\|_{L^p(\Omega)}^p+h^{-p}\|w-w_h\|_{L^p(\Omega)}^p\Big).
\end{align}

By Lemma \ref{PointWiseLemma} we have for any $z\in \overline{\Omega}$ and $v_h\in V_h$,
\begin{align*}
|(w-w_h)(z)|^p\les h^{p-n} |\log h|^{p \bar{s}(p)} \|w-v_h\|^p_{W^{1,p}_h(\Omega),z,s},
\end{align*}
where $\bar{s}(p) = 1$ if $k = s+1-n/p$
and $\bar{s}(p) = 0$ for $k>s+1-n/p$.
Integrating this expression with respect to $z$ yields
\begin{align}\label{asdf}
\|w-w_h\|_{L^p(\Omega)}^p\les h^{p-n} |\log h|^{p \bar{s}(p)} \int_\Omega \|w-v_h\|^p_{W^{1,p}_h(\Omega),z,s}\dx[z].
\end{align}
If $k=1$, then we set $s = n/p$, so that $\bar{s}(p) = 1$, and by \eqref{Fubini3} with $v_h = I_h w$,
\begin{align}\label{asdf1}
\|w-w_h\|_{L^p(\Omega)}^p
&\les h^{p-n} |\log h|^p \int_\Omega \|w-v_h\|^p_{W^{1,p}_h(\Omega),z,n/p}\dx[z]\\
&\nonumber \les h^{2p} |\log h|^{p+1}  \|w\|_{W^{2,p}(\Omega)}^p.
\end{align}
On the other hand, if $k\geq 2$, then we choose $s$ such 
that $n/p<s<k-1+n/p$.  Then $\bar{s}(p)=0$, and by \eqref{asdf} and \eqref{Fubini2},
\begin{align}\label{asdf2}
\|w-w_h\|_{L^p(\Omega)}^p
& \les h^{2p} \|w\|^p_{W^{2,p}(\Omega)}.
\end{align}
Combining \eqref{asdf0} with \eqref{W1pEstimatePart1}, \eqref{asdf1}
and \eqref{asdf2} then yields,
\begin{align}\label{asdf3}
\sum_{e\in \mce} h_e^{1-p} \|[w-w_h]\|_{L^p(e)}^p\les |\log h|^{p+1}  h^p \|w\|^p_{W^{2,p}(\Omega)}.  
\end{align}

Finally combining \eqref{W1pEstimatePart1}, \eqref{asdf3} and applying standard scaling arguments 
yields \eqref{W1pTheoremMainLine}.  This completes the proof of Theorem \ref{W1pTheorem} in the case $p\geq 2$.

\subsubsection{Proof of Theorem \ref{W1pTheorem} for ${\bf 1<}{\bm p}{\bf <2}$}
The proof of $W^{1,p}$ error estimates
in the range $p\in (1,2)$ is based on the following result.
\begin{lemma}\label{infsupAuxLem}
There holds, for $p^\prime\in [2,\infty)$,
\begin{align*}
\|v_h\|_{W^{1,p^\prime}_h(\Omega)}\les |\log h|^{t'} \sup_{0\neq z_h\in V_h} \frac{a_{0,h}(v_h,z_h)}{\|z_h\|_{W^{1,p}_h(\Omega)}}\qquad \forall v_h\in V_h,
\end{align*}
where $p\in (1,2]$ satisfies $1/p+1/p^\prime = 1$ and $t' = (p'+1)/p'$ if $k=1$ and $t'=0$ for $k\geq 2$.
\end{lemma}
\begin{proof}
For a fixed $v_h\in V_h$, let $v\in H^1_0(\Omega)$ satisfy $\mathcal{L}_0 v = \mathcal{L}_{0,h} v_h$
in $\Omega$.  Then $v\in W^{2,p^\prime}(\Omega)$ with
\begin{align}\label{CZOnceAgain}
\|v\|_{W^{2,p^\prime}(\Omega)}\les \|\mathcal{L}_{0,h} v_h \|_{L^{p^\prime}(\Omega)}.
\end{align}
Moreover, due to the definition of $\mathcal{L}_{0,h}$ and the consistency of $a_{0,h}(\cdot,\cdot)$, we find that
\begin{align*}
a_{0,h}(v,z_h) = a_{0,h}(v_h,z_h)\qquad \forall z_h\in V_h.
\end{align*}

Since $p^\prime\geq 2$, we can apply the results of the previous section to conclude that
\begin{align}\label{infsupAuxLemLine1}
\|v_h\|_{W^{1,p^\prime}_h(\Omega)}\les |\log h|^{t'} \big(\|v\|_{W^{1,p^\prime}(\Omega)}+h \|v\|_{W^{2,p^\prime}(\Omega)}\big).
\end{align}
Denote by $\mathcal{P}_h:L^2(\Omega)\to V_h$ the $L^2$ projection onto $V_h$.  We then write
\begin{align*}
\|v\|_{W^{1,p^\prime}(\Omega)}
&\les \sup_{z\in W^{1,p}_0(\Omega)} \frac{(A_0 \nab v,\nab z)}{\|z\|_{W^{1,p}(\Omega)}}
= \sup_{z\in W^{1,p}_0(\Omega)} \frac{(\mathcal{L}_0 v,z)}{\|z\|_{W^{1,p}(\Omega)}} 
= \sup_{z\in W^{1,p}_0(\Omega)} \frac{(\mathcal{L}_{0,h} v_h,\mathcal{P}_hz)}{\|z\|_{W^{1,p}(\Omega)}}.
\end{align*}
Standard arguments show that $\|\mathcal{P}_hz\|_{W^{1,p}_h(\Omega)}\les \|z\|_{W^{1,p}(\Omega)}$ for all $z\in W^{1,p}(\Omega)$;
thus,
\begin{align}\label{infsupAuxLemLine2}
\|v\|_{W^{1,p^\prime}(\Omega)}\les \sup_{0\neq z_h\in V_h} \frac{(\mathcal{L}_{0,h} v_h,z_h)}{\|z_h\|_{W^{1,p}_h(\Omega)}}
=\sup_{0\neq z_h\in V_h} \frac{a_{0,h}(v_h,z_h)}{\|z_h\|_{W^{1,p}_h(\Omega)}}.
\end{align}

Likewise, using \eqref{CZOnceAgain}, \eqref{norm_equiv} and an inverse estimate yields 
\begin{align}\label{infsupAuxLemLine3}
 \|v\|_{W^{2,p^\prime}(\Omega)}\les \|\mathcal{L}_{0,h} v_h\|_{L^{p^\prime}_h(\Omega)}
  = \sup_{0\neq z_h\in V_h} \frac{a_{0,h}(v_h,z_h)}{\|z_h\|_{L^{p}(\Omega)}}
  \les h^{-1} \sup_{0\neq z_h\in V_h} \frac{a_{0,h}(v_h,z_h)}{\|z_h\|_{W^{1,p}(\Omega)}}.
  \end{align}
  Applying the estimates \eqref{infsupAuxLemLine2}--\eqref{infsupAuxLemLine3}
  to \eqref{infsupAuxLemLine1} then gives the desired result. 
\end{proof}

We now prove Theorem \ref{W1pTheorem} for $1<p<2$.
To this end, for $w_h\in V_h$ and $w\in W^{2,p}(\Omega)\cap W^{1,p}_0(\Omega)$
satisfying \eqref{ErrorEquation},  let $v_h\in V_h$ be the unique solution to
\begin{align*}
a_{0,h}(v_h,z_h) = \int_{\Omega} |\nab_h w_h|^{p-2}  \nab_h w_h\cdot \nab_h z_h\dx[x]
+\sum_{e\in \mce} h_e^{1-p} \int_e |[w_h]|^{p-2} [w_h][z_h]\dx[S] 
\end{align*}
for all $z_h\in V_h$. Setting $z_h = w_h$ and using a scaling argument yields
\begin{align*}
\|w_h\|_{W^{1,p}_h(\Omega)}^p \les a_{0,h}(v_h,w_h).
\end{align*}
Moreover, Lemma \ref{infsupAuxLem} and H\"older's inequality gets
\begin{align*}
\|v_h\|_{W^{1,p^\prime}_h(\Omega)}\les |\log h|^{t'} \|w_h\|_{W^{1,p}_h(\Omega)}^{p-1}.
\end{align*}
Consequently,
\begin{align*}
 \|w_h\|_{W^{1,p}_h(\Omega)} 
&= \frac{\|w_h\|^p_{W^{1,p}_h(\Omega)}}{\|w_h\|^{p-1}_{W^{1,p}_h(\Omega)}}
\les |\log h|^{t'} \frac{a_{0,h}(v_h,w_h)}{\|v_h\|_{W^{1,p^\prime}_h(\Omega)}}\\
&= |\log h|^{t'} \frac{a_{0,h}(w,v_h)}{\|v_h\|_{W^{1,p^\prime}_h(\Omega)}}
\les{|\log h|^{t'}}\|w\|_{W^{1,p}_h(\Omega)}.
\end{align*} 
Standard arguments then show that this estimate implies
\begin{align}\label{W1pEstimatePart4}
 \|w-w_h\|_{W^{1,p}_h(\Omega)} \les |\log h|^{t'} h\|w\|_{W^{2,p}(\Omega)}\quad 1<p<2.
 \end{align}
This completes the proof of Theorem \ref{W1pTheorem}
upon noting that $t' = (p'+1)/p' = (2p-1)/p \le (p+1)/p = t$
for $p\in (1,2]$.

\subsection{DG discrete Calderon-Zygmund estimates for PDEs with constant coefficients}\label{section3.2}
The {goal of this subsection is to establish} a stability result for the operator 
$\L_{0,h}^{\e}$ in the $W_h^{2,p}$-norm, which is a discrete counterpart of (\ref{eqn2:9}). Such an estimate 
can be regarded as a DG discrete Calderon-Zygmund estimate for $\L_{0,h}^{\e}$.  

{
\begin{theorem} \label{lem2:7} 
(i) For $\e=1$ and $1< p< \infty$ we have
\begin{align}
\|w_h\|_{W_h^{2,p}(\W)}\lss |\log h|^t  \|\L^\e_{0,h} w_h\|_{L^p(\W)} \qquad \forall w_h\in V_h, \label{eqn2:10}
\end{align}
where $t = (p+1)/p$ if $k=1$ and $t=0$ if $k\geq 2$.

\medskip
(ii) \eqref{eqn2:10} also holds with $t=0$ for $\e \in \{1,0,-1\}$ and $p=2$.
\end{theorem}
}

\medskip
\begin{proof}
(i) We observe that \eqref{eqn2:10} is equivalent to showing
\begin{align}\label{eqn:2.40}
\|(\L_{0,h}^{\e})^{-1}\phi_h\|_{\Wh(\W)}\lss { |\log h|^t} \|\phi_h\|_{L^p(\W)} \qquad\forall\phi_h\in V_h.
\end{align}

For any $\phi_h\in V_h$, let $w:=\L_0^{-1}\phi_h\in W^{2,p}(\W)\cap W_0^{1,p}(\W)$ and 
$w_h:=(\L_{0,h}^{\e})^{-1}\phi_h\in V_h$.  Since $w\in W^{2,p}(\W)\cap W_0^{1,p}(\W)$ we have
\begin{alignat*}{2}
a_{0,h}^{\e}(w,v_h) &= (\phi_h,v_h) = a_{0,h}^{\e}(w_h,v_h)\qquad \forall v_h\in V_h.
\end{alignat*}
Thus $w_h$ is the IP-DG approximate solution to $w$. 
Applying Theorem \ref{W1pTheorem} 
and the elliptic regularity estimate, we obtain
\begin{align}\label{W1pEstimateLineABC}
\|w-w_h\|_{W^{1,p}_h(\Omega)}\les |\log h|^t h \|w\|_{W^{2,p}(\Omega)}\les |\log h|^t h\|\phi_h\|_{L^p(\W)}.
\end{align}
Moreover, by Lemma \ref{lem2:3.3} and the Calderon-Zygmund estimate for $\L_0$ we have
\begin{equation}\label{CZ_estimate}
\|w\|_{W^{2,p}_h(\W)} \leq \|w\|_{W^{2,p}(\W)} \lss \|\varphi_h\|_{L^p(\W)}.
\end{equation}

Denote by $I_h:C^0(\Omega)\to V_h$ the nodal interpolation operator onto $V_h$.  
By  finite element interpolation theory \cite{SIAM:Ciarlet} we have
\begin{align}\label{StandardNodalLine}
h^{-1} \|w-I_h w\|_{W^{1,p}_h(\Omega)}+  \|w-I_h w\|_{W^{2,p}_h(\Omega)}\les \|w\|_{W^{2,p}(\Omega)}.
\end{align}
Therefore by the triangle inequality, an inverse estimate, Lemma  \ref{lem2:3.3},  \eqref{W1pEstimateLineABC},
and \eqref{CZ_estimate}, we obtain
\begin{align*}
\|w_h\|_{W^{2,p}_h(\Omega)}
&\le \|w- I_h w\|_{W^{2,p}_h(\Omega)} +\|I_h w-w_h\|_{W^{2,p}_h(\Omega)}
+\|w\|_{W^{2,p}_h(\Omega)}\\
&\lss h^{-1} \|I_h w-w_h\|_{W^{1,p}_h(\Omega)}+ \|\varphi_h\|_{L^p(\Omega)}\\
&\le h^{-1}\big(\|w-w_h\|_{W^{1,p}_h(\Omega)}+\|w-I_h w\|_{W^{1,p}_h(\Omega)}\big) + \|\varphi_h\|_{L^p(\Omega)}\\
&\lss |\log h|^t \|\varphi_h\|_{L^p(\Omega)} = |\log h|^t \|\mathcal{L}_{0,h} w_h \|_{L^p(\Omega)}.
\end{align*}

(ii) The proof of this part is exactly same as that of Part (i), the only difference is that 
now \eqref{W12ErrorEstimate}, instead of \eqref{W1pTheoremMainLine},  should be called in the proof. 
\end{proof}

\section{IP-DG methods and their convergence analysis} \label{section4}

As mentioned in Section \ref{section1}, our primary goal in this paper to develop convergent IP-DG methods
for approximating the $W^{2,p}$ strong solution to the boundary value problem \eqref{NonDivProblem}.
We assume that equation \eqref{eqn1:1} is uniformly elliptic, precisely, we assume 
$A\in \left[C^0(\overline{\W})\right]^{n\times n}$ is positive definite, that is,
there exist constants $\Lambda>\lambda>0$ such that
\begin{align} \label{eqn3:01}
\lambda |\xi|^2\leq A(x) \leq \Lambda |\xi|^2 \quad\text{ for all } \xi\in\R^n \text{ and } x\in\overline{\W}.
\end{align}
We also assume that the solution $u$ satisfies the following Calderon-Zygmund estimate:
\begin{align} \label{eqn3:02}
\|u\|_{W^{2,p}(\W)}\lss \|f\|_{L^p(\W)}.
\end{align}
It is well-known \cite[Chapter 9]{Sp:DT} that the above estimate holds for any $f\in L^{p}(\W)$ with $1< p <\infty$ 
if $\partial\W\in C^{1,1}$. Moreover, when $n,p=2$, \eqref{eqn3:02} also holds if $\Omega$ is a convex domain 
\cite{SIAM:BCO,Sp:Br}.  

\subsection{Formulation of IP-DG methods}\label{section4.1}

We follow the same recipe as in the constant coefficient case to build
our IP-DG methods. To this end, we momentarily assume $A\in [C^1(\W)]^{n\times n}$, so that we can rewrite the 
PDE \eqref{eqn1:1} in divergence form as follows:
\begin{align}
-\grad\cdot(A\grad u) + (\grad\cdot A)\cdot \grad u = f, \label{eqn4}
\end{align}
where $\grad\cdot A$ is defined row-wise. We then define the following (standard) IP-DG methods 
for problem \eqref{eqn4} by seeking $u_h\in V_h$ such that 
\begin{align}\label{eqn:3.4}
&\int_{\Omega} (A\grad_h u_h)\cdot\grad_h v_h \dx[x] + \int_{\Omega}((\grad \cdot A)\cdot\grad_h u_h)v_h \dx[x] \\
&\qquad-\sum_{e\in\E}\int_{e}\{ A\grad u_h\dv\} [v_h] \dx[S]-\e \sum_{e\in\E}\int_{e}\{ A\grad v_h\dv\} [u_h] \dx[S] \nonumber \\
&\qquad+\sum_{e\in\E} \int_e \frac{\ge}{h_e}[u_h][v_h]\dx[S] 
= \int_{\W} fv_h \dx[x], \nonumber
\end{align}
where $\ge\geq\gamma_*(\|A\|_{L^\infty(\Omega)},\mathcal{T}_h)>0$.
We emphasize that $\gamma^*$ is independent of the derivatives of $A$.

Now come back to our case in hand with $A\in [C^0(\W)]^{n\times n}$. Clearly, the term $\grad\cdot A$ does 
not exist as a function (it is in fact a Radon measure), so the above formulation is not defined for 
the case we are considering. To overcome this difficulty, our idea is to apply the DG integration by parts formula
\eqref{eqn2:8} to the first term on the left-hand side of \eqref{eqn:3.4}, yielding 
\begin{align}
a_h(w_h,v_h) :=& -\int_{\W}(A:D_h^2w_h)v_h \dx[x] + \sum_{e\in\EI}\int_e [ A\grad w_h\dv] \{v_h\} \dx[S] \label{eqn6} \\
&-\e \sum_{e\in\E}\int_{e}\{ A\grad v_h\dv\} [w_h] \dx[S] +\sum_{e\in\E}\int_e \frac{\ge}{h_e}[w_h][v_h] \dx[S]. \nonumber 
&  
\end{align}
No derivative of $A$ appears in the above new form of $a_h^{\e}(\cdot,\cdot)$; thus, it is well-defined on $V_h\times V_h$.
This leads to the following definition.

\begin{defn}
Our IP-DG methods are defined by seeking $u_h\in V_h$ such that 
\begin{align} \label{eqn:dis}
a^\e_h(u_h,v_h) = (f,v_h) \qquad\forall v_h\in V_h, \quad \e\in \{1,0,-1\}.
\end{align}
\end{defn}
When $\e=1$ we will refer to the method as ``symmetrically induced'' even though the bilinear form is  
not symmetric. Likewise, $\e=0$ and $\e=-1$ yields an ``incompletely induced'' and ``non-symmetrically induced'' method,
respectively.

\subsection{Stability analysis}\label{section4.2} 
As in Section \ref{sec:NewConstant} we can define the IP-DG approximation $\L_h^{\e}$ of $\L$ on $V_h$ using the bilinear 
form $a^\e_h(\cdot,\cdot)$; precisely, we define $\L^\e_h:V_h\to V_h$ by
\begin{align}
\bigl(\L^\e_h w_h,v_h\bigr):=a^\e_h(w_h,v_h) \qquad\forall w_h,v_h\in V_h.
\end{align}
Since we can extend the domain of $a_h^\e(\cdot,\cdot)$ to $W^{2,p}(\Th)\times W^{2,p'}(\Th)$, then the domain and co-domain 
of $\L_h$ can be extended to the broken Sobolev spaces $W^{2,p}(\Th)$ and $(W^{2,p'}(\Th))^*$ respectively.  

The goal of this subsection is to establish a DG discrete Calderon-Zygmund estimate similar
to \eqref{eqn2:10} for the operator $\L^\e_h$.  {To this end, our main idea  is to 
mimic, at the discrete level, the ``freezing the coefficients'' technique and the covering argument found in Schauder 
theory and $W^{2,p}$ strong solution theory \cite[Chapters 6 and 9]{Sp:DT}. Since $A$ is continuous, 
we show that in a small ball $B_{\d} (\subset \W)$ 
$A$ behaves as if it were constant. This allows us to conclude that $\L^\e_h$ is locally very close
to $\L^\e_{0,h}$ in the ball $B_{\d}(x_0)$ for any $x_0\in \W$.
By applying the above mentioned ``freezing the coefficients'' technique and covering
argument to the formal adjoint of $\L^\e_h$, we are able to prove a global left-side inf-sup 
condition for $\L^\e_h$. Then by employing a duality argument, we derive the desired 
discrete Calderon-Zygmund estimate for $\L^\e_h$.}

We now proceed to establish a few auxiliary lemmas which will be needed to show the desired estimate.

\begin{lemma} \label{lem3:1}
For all $\d>0$, there exists $R_{\d}>0$ and $h_{\d}>0$ such that for all $x_0\in\W$ and $A_0\equiv A(x_0)$
\begin{align}
\|(\L_h^{\e}-\L_{0,h}^{\e})w\|_{L_h^p(B_{R_{\d}}(x_0))}
\lss \d \|w\|_{\Wh(B_{R_{\d}}(x_0))} \quad\forall w\in W^{2,p}(\Th),\,\forall h\leq h_{\d}.
\end{align}
Here, $B_{R_{\d}}(x_0):=\{x\in \Omega:\ |x-x_0|<R_{\delta}\}$
denotes the ball with center $x_0$ and radius $R_{\d}$.
\end{lemma}

\begin{proof}
Since $A$ is continuous on $\overline{\W}$, then it is uniformly continuous.  Therefore, for every 
$\d>0$ there exists $R_{\d}>0$ such that if $x,y\in\W$ satisfies $|x-y|<R_{\d}$, we have $|A(x)-A(y)|<\d$.  
Consequently, for any $x_0\in\W$ 
\begin{align}
\|A-A_0\|_{L^{\infty}(B_{R_{\d}})}\leq\d,
\end{align}
where we have used the shorthand notation $B_{R_{\d}}:=B_{R_{\d}}(x_0)$.

Set $h_{\d}=\min\{h_0,\frac{R_{\d}}{4}\}$ and let $0<h<h_{\d}$, $w\in W^{2,p}(\Th)$, and $v_h\in V_h(B_{R_{\d}})$.  
Since $(\L_{0,h}^{\e}-\L_h^{\e})w\in W^{2,p}(\Th)$, it follows from \eqref{eqn5} and \eqref{eqn6} that for every 
$v_h\in V_h(B_{R_{\d}})$ we have
\begin{align*}
&\bigl((\L_{0,h}^{\e}-\L_h^{\e})w,v_h \bigr) 
=-\int_{\Omega\cap B_{R_{\d}}}((A_{0}-A):D^2_hw)v_h \dx[x] \\
&\quad +\sum_{e\in\EI}\int_{e\cap\overline{B}_{R_{\d}}}[(A_{0}-A)\grad w\dv] \{v_h\} \dx[S] 
%
-\e\sum_{e\in\E}\int_{e\cap\overline{B}_{R_{\d}}}\{ (A-A_0)\grad v_h\dv\} [w_h] \dx[S] \\
&\leq \|A-A_{0}\|_{L^{\infty}(B_{R_{\d}})} \left(\|D^2_hw\|_{L^p(\Omega\cap B_{R_{\d}})}
\|v_h\|_{L^{p'}(\Omega\cap B_{R_{\d}})} \right. \\
&\quad + \Bigl(\sum_{e\in\E}h_e^{1-2p}\|[w]\|_{L^p(e\cap\overline{B}_{R_{\d}}))}^p \Bigr)^\frac{1}{p}
\Bigl(\sum_{e\in\E}h_e h_e^{p'}\|\{\grad v_h\dv\}\|_{L^{p'}(e\cap\overline{B}_{R_{\d}}))}^{p'} \Bigr)^\frac{1}{p'} \\
%
%
&\quad \left. + \Bigl(\sum_{e\in\EI}h_e^{1-p}\|[\grad w]\|_{L^{p}(e\cap\overline{B}_{R_{\d}})}^p\Bigr)^{\frac{1}{p}}
\Bigl(\sum_{e\in\EI}h_e\|\{v_h\}\|_{L^{p'}(e\cap\overline{B}_{R_{\d}})}^{p'} \Bigr)^{\frac{1}{p'}}\right)\\
&\lss \|A-A_{0}\|_{L^{\infty}(B_{R_{\d}})}\|w\|_{\Wh(B_{R_{\d}})}\Bigl(\|v_h\|_{L^{p'}(B_{R_{\d}})}
+h\|\grad_h v_h\|_{L^{p'}(B_{R_{\d}})}\Bigr) \\
&\lss \d \|w\|_{\Wh(B_{R_{\d}})}\|v_h\|_{L^{p'}(B_{R_{\d}})} 
= \d \|w\|_{\Wh(B_{R_{\d}})}\|v_h\|_{L^{p'}(B_{R_{\d}})}.
\end{align*}
Dividing both sides by $\|v_h\|_{L^{p'}(B_{R_{\d}})}$ yields the desired estimate.  The proof is complete.
\end{proof}

The next lemma shows that $\L^\e_h$ is locally a bounded operator on $W^{2,p}(\Th)$.
\begin{lemma} \label{lem3:3} 
For any $x_0\in \W$ and $R\geq h$, there holds
\begin{align}
\|\L_h^{\e} w\|_{L_h^p(B_R(x_0))}\lss \|w\|_{\Wh(B_{R}(x_0))}\qquad\forall w\in W^{2,p}(\Th). 
\end{align}
\end{lemma}
\begin{proof}
Set $B_R:= B_{R}(x_0)$ and let  $v_h\in V_h(B_R)$. For $e\in\E$, set $e_R:= e\cap\overline{B}_R$.  
By the trace estimate \eqref{eqn2:12}), the definition of $\L_h^{\e}$, and an inverse inequality we have
\begin{align*}
\bigl(\L_h^{\e}w,v_h\bigr)
&=-\int_{\Omega\cap B_{R}}(A:D^2_hw)v_h \dx[x]
+\sum_{e\in\EI}\int_{e_R}[A\grad w\dv] \{v_h\} \dx[S] \\ 
&\quad -\e\sum_{e\in\E}\int_{e_R}\{A\grad v_h\dv\}[w]\dx[S] 
%
+\sum_{e\in\E}\int_{e_R}\frac{\ge}{h_e}[w][v_h]\dx[S]\\ 
&\lss \|D^2w\|_{L^p(\Omega\cap B_{R})}\|v_h\|_{L^{p'}(\Omega\cap B_{R})}\\
&\quad +\Bigl(\sum_{e\in\EI}h_e^{1-p}\|[\grad w]\|_{L^{p}(e_R)}^p \Bigr)^{\frac{1}{p}}
\Bigl(\sum_{e\in\EI}h_e\|\{v_h\}\|_{L^{p}(e_R)}^{p'} \Bigr)^{\frac{1}{p'}} \\ 
&\quad +\Bigl(\sum_{e\in\E}\ge^p h_e^{1-2p}\|[w]\|_{L^{p}(e_R)}^p  \Bigr)^{\frac{1}{p}}
\Bigl(\sum_{e\in\E}h_e^{p'}h_e\|\{\grad_h v_h\dv\}\|_{L^{p'}(e_R)}^{p'} \Bigr)^{\frac{1}{p'}} \\
%
%
&\quad +\Bigl(\sum_{e\in\E}\ge^p h_e^{1-2p}\|[w]\|_{L^{p}(e_R)}^p  \Bigr)^{\frac{1}{p}}
\Bigl(\sum_{e\in\E}h_e\|[v_h]\|_{L^{p'}(e_R)}^{p'} \Bigr)^{\frac{1}{p'}} \\
%
%
%
&\lss \|w\|_{\Wh(B_{R})}\|v_h\|_{L^{p'}(B_{R})}.
\end{align*}
Dividing both sides by $\|v_h\|_{L^{p'}(B_{R})}$ yields the desired estimate.  
\end{proof}

Our last lemma establishes a left-side inf-sup condition for $\L_h^{\e}$.  This estimate relies on the 
formal adjoint operator $\aL:=(\L_h^{\e})^*$ and some techniques from \cite{MC:SW}.  
\begin{lemma} \label{lem3:5} %
There exists an $h_0>0$ such that for all $h\leq h_0 $ and $k\geq 2$ we have
\begin{align}\label{eqn3:30}
\|v_h\|_{L^{p'}(\W)}\lss \sup_{0\neq w_h\in V_h}\frac{(\L_h^{\e} w_h,v_h)}{\|w_h\|_{\Wh(\W)}} \qquad\forall 
v_h\in V_h,
\end{align}
where $1<p<\infty$ if $\e=1$ and $p=2$ if $\e \in \{0,-1\}$.
\end{lemma}
\begin{proof}
Note that \eqref{eqn3:30} is equivalent to
\begin{align}\label{eqn3:30.5}
\|v_h\|_{L^{p'}(\W)} &\lss \sup_{0\neq w_h\in V_h}\frac{(\L_h^{\e} w_h,v_h)}{\|w_h\|_{\Wh(\W)}} 
= \sup_{0\neq w_h\in V_h}\frac{(\aL v_h,w_h)}{\|w_h\|_{\Wh(\W)}} = \|\aL v_h\|_{W_h^{-2,p'}(\W)}
\end{align}
for all $v_h\in V_h$.  We divide the remaining proof into three steps.

\medskip
{\em Step 1: Local estimates.}
Let $x_0\in\W,\ A_0 \equiv A(x_0)$, $\d_0$, $h_{\delta_0}$, $R_{\d_0}$, $R_1:=(1/3)R_{\d_0}$, and $B_1:=B_{R_1}(x_0)$ be 
as in Lemma \ref{lem3:1} with $\delta_0>0$ to be determined, and set $h\le h_{\delta_0}$.

By the elliptic regularity of $\L$, 
for any $v_h\in V_h(B_1)$, there exists $\phi\in W^{2,p}(\W)\cap W_0^{1,p}(\W)$ such that $\L\phi = v_h|v_h|^{p-2}$ in $\W$ and satisfies the estimate
\begin{align}\label{eqn3:32}
\|\phi\|_{W^{2,p}(\W)} \lss \|v_h\|_{L^{p'}(\W)}^{p'-1} = \|v_h\|_{L^{p'}(B_1)}^{p'-1}.
\end{align}
Since $\L_h^{\e}$ is consistent with $\L$ for any $\phi_h\in V_h$ we have
\begin{align}\label{eqn3:33}
\|v_h\|_{L^{p'}(B_1)} &= \|v_h\|_{L^{p'}(\W)} = (\L\phi,v_h) = (\L_h^{\e}\phi,v_h) \\ 
&= (\L_h^{\e}\phi_h,v_h) + \big(\L_h^{\e}(\phi-\phi_h),v_h\big) \nonumber \\
&= (\aL v_h,\phi_h) + \big(\L_{0,h}^{\e}(\phi-\phi_h),v_h\big) + \big((\L_h^{\e}-\L_{0,h}^{\e})(\phi-\phi_h),v_h\big). \nonumber
\end{align}
From the existence-uniqueness of the IP-DG scheme \eqref{eqn2:7}, there exists $\phi_h\in V_h$ such that 
\begin{align*}
\big(\L_{0,h}^{\e}(\phi-\phi_h),w_h\big) = 0 \quad\forall w_h\in V_h.
\end{align*}
Combining Galerkin orthogonality, Theorem \ref{lem2:7}, and \eqref{eqn3:32} gives us the solution estimate
\begin{align}\label{eqn3:34}
\|\phi_h\|_{\Wh(\W)} \lss \|\L_{0,h}^{\e}\phi_h\|_{L_h^p(\W)} = \|\L_{0,h}^{\e}\phi\|_{L_h^p(\W)} \lss \|\phi \|_{W^{2,p}(\W)} \lss \|v_h\|_{L^{p'}(B_1)}^{p'-1}.
\end{align}
Using Lemma \ref{lem3:1} and \eqref{eqn3:32}--\eqref{eqn3:34} we have
\begin{align*}
\|v_h\|_{L^{p'}(B_1)}^{p'} &= (\aL v_h,\phi_h) + ((\L_h^{\e}-\L_{0,h}^{\e})(\phi-\phi_h),v_h) \\
&\leq \|\aL v_h\|_{W_h^{-2,p'}(\W)}\|\phi_h\|_{\Wh(\W)}+ \|(\L_h^{\e}-\L_{0,h}^{\e})(\phi-\phi_h)\|_{L_h^p(B_1)}\|v_h\|_{L^{p'}(B_1)} \\
&\lss \|\aL v_h\|_{W_h^{-2,p'}(\W)}\|v_h\|_{L^{p'}(B_1)}^{p'-1}+ \d_0\|\phi-\phi_h\|_{\Wh(B_1)}\|v_h\|_{L^{p'}(B_1)} \\
&\lss \|\aL v_h\|_{W_h^{-2,p'}(\W)}\|v_h\|_{L^{p'}(B_1)}^{p'-1}+ \d_0\|v_h\|_{L^{p'}(B_1)}^{p'}.
\end{align*}
Taking $\d_0$ sufficiently small to move the right hand term to the left side and dividing by $\|v_h\|_{L^{p'}(B_1)}^{p'-1}$ gives us 
the  local estimate
\begin{align}\label{eqn3:35}
\|v_h\|_{L^{p'}(B_1)}\lss \|\aL v_h\|_{W_h^{-2,p}(B_1)}\qquad \forall v_h\in V_h(B_1).
\end{align}

\medskip
{\em Step 2: A G\"{a}rding type inequality {by a covering argument}.}
Given $R_1$ from Step 1, let $R_2 = 2R_1$ and $R_3=3R_1$.  Let $\eta\in C^3(\W)$ 
be a cutoff function satisfying
\begin{align} \label{eqn3:20}
0\leq \eta\leq 1, \quad{\eta\big|_{B_1}=1}, \quad\eta\big|_{\W\setminus B_2}=0, \quad|\eta|_{W^{m,\infty}(\W)}=O(R_1^{-m}).
\end{align}
For any $v_h\in V_h$, we have by \eqref{eqn3:35},
\begin{align}\label{eqn3:36}
&\|v_h\|_{L^{p'}(B_1)} 
= \|\eta v_h\|_{L^{p'}(B_1)} \leq \|\eta v_h-\tI(\eta v_h)\|_{L^{p'}(B_1)}+ \|\tI(\eta v_h)\|_{L^{p'}(B_1)}\\
&\lss \|\eta v_h-\tI(\eta v_h)\|_{L^{p'}(B_1)}+ \|\aL(\tI(\eta v_h))\|_{W_h^{-2,p'}(B_1)}\nonumber\\
&\lss \|\eta v_h-\tI(\eta v_h)\|_{L^{p'}(B_1)}+ \|\aL(\tI(\eta v_h)-\eta v_h)\|_{W_h^{-2,p'}(B_1)} + \|\aL(\eta v_h)\|_{W_h^{-2,p'}(B_1)}.\nonumber
\end{align}
We now bound the second term on the right hand side of 
\eqref{eqn3:36}. By the definition of $\|\cdot\|_{W_h^{-2,p}}$, 
Lemma \ref{lem3:3} and \eqref{norm_equiv}, for any $w_h\in V_h$ we have
\begin{align*}
\|\aL(&\tI(\eta v_h)-\eta v_h)\|_{W_h^{-2,p'}(B_1)} 
= \sup_{0\neq w_h\in V_h}\frac{(\aL(\tI(\eta v_h)-\eta v_h),w_h)}{\|w_h\|_{\Wh(B_1)}} \\
&\leq \sup_{w_h\in V_h}\frac{(\L_h^{\e} w_h,\tI(\eta v_h)-\eta v_h)}{\|w_h\|_{\Wh(B_1)}}\lss \sup_{w_h\in V_h}\frac{\|\L_h^{\e} w_h\|_{L_h^p(B_1)}\|\tI(\eta v_h)-\eta v_h)\|_{L^{p'}(B_1)}}{\|w_h\|_{\Wh(B_1)}}\nonumber \\
&\lss\sup_{w_h\in V_h}\frac{\|w_h\|_{\Wh(B_1)}\|\tI(\eta v_h)-\eta v_h)\|_{L^{p'}(B_1)}}{\|w_h\|_{\Wh(B_1)}} = \|\tI(\eta v_h)-\eta v_h\|_{L^{p'}(B_1)}. \nonumber
\end{align*}
Thus \eqref{eqn3:36} becomes
\begin{align}\label{eqn3:37}
\|v_h\|_{L^{p'}(B_1)}\lss \|\eta v_h-\tI(\eta v_h)\|_{L^{p'}(B_1)}+ \|\aL(\eta v_h)\|_{W_h^{-2,p'}(B_1)}.
\end{align}
Using Lemmas \ref{lem2:5.5}, \ref{lem2:4}, and \ref{lem3:3} with \eqref{eqn3:37} yields
\begin{align}\label{eqn3:38}
\|v_h\|_{L^{p'}(B_1)}&\lss \frac{h}{R_1}\|v_h\|_{L^{p'}(B_3)}+ \|\aL(\eta v_h)\|_{W_h^{-2,p'}(B_3)} \\
&\lss \frac{1}{R_1}\|v_h\|_{W^{-1,p'}(B_3)}+ \|\aL(\eta v_h)\|_{W_h^{-2,p'}(B_3)}. \nonumber
\end{align}
We now {want} to remove the cutoff function $\eta$ from the adjoint operator
appearing in the right-hand side of \eqref{eqn3:38}.  For $w_h\in V_h(B_3)$,
we break up $\aL( \eta v_h)$ as follows:
\begin{align}\label{eqn3:39}
(\aL&( \eta v_h),w_h) = (\L_h^{\e} w_h,\eta v_h) = (\L_h^{\e} w_h\eta,v_h)+\bigg[(\L_h^{\e} w_h,\eta v_h)-(\L_h^{\e} w_h\eta, v_h)\bigg] \\
&= (\L_h^{\e} (\tI(w_h\eta)),v_h)+(\L_h^{\e} (w_h\eta-\tI(w_h\eta)),v_h) \nonumber \\
&\quad+\bigg[(\L_h^{\e} w_h,\eta v_h)-(\L_h^{\e} w_h\eta, v_h)\bigg]\nonumber \\
&=: I_1 + I_2 + I_3 \nonumber.
\end{align}

We {then}  seek to bound each $I$ in order. To bound $I_1$, we will use the definition of 
$\|\cdot\|_{W_h^{-2,p}}$, the stability of $\tI$, and Lemma \ref{lem2:55} to obtain
\begin{align}\label{eqn3:41}
I_1 &= (\aL v_h, \tI(w_h\eta))  
 \lss \|\aL v_h\|_{W_h^{-2,p'}(B_3)}\|\tI(\eta w_h)\|_{\Wh(B_3)} \\
&\lss \|\aL v_h\|_{W_h^{-2,p'}(B_3)}\|\eta w_h\|_{\Wh(B_3)}
\lss \frac{1}{R_1^2}\|\aL v_h\|_{W_h^{-2,p'}(B_3)}\|w_h\|_{\Wh(B_3)}. \nonumber
\end{align}
For $I_2$ we use Lemmas \ref{lem2:5.5}, \ref{lem2:4}, \ref{lem3:3} to get
\begin{align}\label{eqn3:42}
I_2 &= (\L_h^{\e} (w_h\eta-\tI(w_h\eta)),v_h)
\lss \| w_h\eta-\tI(w_h\eta) \|_{\Wh(B_3)}\|v_h\|_{L^{p'}(B_3)} \\
&\lss \frac{h}{R_1^3}\| w_h \|_{\Wh(B_3)}\|v_h\|_{L^{p'}(B_3)} \lss \frac{1}{R_1^3}\| w_h \|_{\Wh(B_3)}\|v_h\|_{W^{-1,p'}(B_3)}.  \nonumber
\end{align} 
To bound $I_3$ we  introduce the operator $\L_{0,h}^{\e}$ .  
For $e\in \E$ let $e_3:= e\cap B_3$, and define $\tilde{A} := A-A_0$.
We then write
\begin{align}\label{eqn3:43}
I_3 &= (\L_h^{\e} w_h,\eta v_h)-(\L_h^{\e} w_h\eta, v_h) \\ 
& =  (\L_{0,h}^{\e} w_h,\eta v_h)-(\L_{0,h}^{\e} w_h\eta, v_h) + \nonumber\\
&\quad\quad+ \big[ (\L_h^{\e} w_h,\eta v_h)-(\L_h^{\e} w_h\eta, v_h) - (\L_{0,h}^{\e} w_h,\eta v_h)+(\L_{0,h}^{\e} w_h\eta, v_h)\big] \nonumber \\
&= -\int_{B_3} \left(w_h A_0:D^2\eta + (A_0+A_0^T)\grad\eta\cdot\grad_h w_h\right) v_h \dx[x]  \nonumber \\
&\quad -\e\sum_{e\in \E} \int_{e_3} (A_0 \grad \eta  \dv)\{v_h\} [w_h]\dx[S]\nonumber\\ 
&\quad -\int_{B_3} \left(w_h (\tilde{A}:D^2\eta + \left(\tilde{A}+\tilde{A}^T\right)\grad\eta\cdot\grad_h w_h\right) v_h \dx[x]\nonumber\\
&\quad -\e \sum_{e\in \E} \int_{e_3} (\tilde{A} \grad \eta\dv)\{v_h\}[w_h]\dx[S]\nonumber
=: K_1 + K_2 + K_3 + K_4. \nonumber
\end{align} 
We now must bound each $K_i$. To bound $K_1$ we use the definition of $\|\cdot\|_{W_h^{-1,p}(B_3)}$ and Lemma \ref{lem2:5.5} to get
\begin{align}\label{eqn3:44}
K_1 &\lss \left(\|w_h A_0:D^2\eta\|_{W_h^{1,p}(B_3)}+ \|(A_0+A_0^T)\grad\eta\cdot\grad_h w_h\|_{W_h^{1,p}(B_3)}\right)\|v_h\|_{W_h^{-1,p'}(B_3)} \\
&\lss \frac{1}{R_1^3}\|w_h\|_{\Wh(B_3)}\|v_h\|_{W_h^{-1,p'}(B_3)}. \nonumber
\end{align}
The bound of $K_2$ uses Lemmas \ref{lem2:2}, \ref{lem2:55} to obtain
\begin{align}\label{eqn3:45}
K_2 &\lss \frac{1}{R_1}\left(\sum_{e\in\E}h_e^{1-2p}\|[w_h]\|_{L^p(e_3)}^p\right)^{\frac{1}{p}}\left(\sum_{e\in\E}h_eh_e^{p'}\|\{v_h\}\|_{L^{p'}(e_3)}^{p'}\right)^{\frac{1}{p'}}\\
&\lss \frac{1}{R_1}\|w_h\|_{\Wh(B_3)}\left(h\|v_h\|_{L^{p'}(B_3)}\right) \nonumber \\
&\lss \frac{1}{R_1}\|w_h\|_{\Wh(B_3)}\|v_h\|_{W_h^{-1,p'}(B_3)}. \nonumber 
\end{align}
We  use similar techniques as \eqref{eqn3:44}, \eqref{eqn3:45} and the fact that $\|\tilde{A}\|_{L^\infty(B_3)} \leq \d_0$ to get
\begin{align}\label{eqn3:46}
K_3 &\lss \left(\|w_h \tilde{A}:D^2\eta\|_{L^p(B_3)}+ \|(\tilde{A}+\tilde{A}^T)\grad\eta\cdot\grad_h w_h\|_{L^p(B_3)}\right)\|v_h\|_{L^{p'}(B_3)} \\
&\lss \d_0\left(\frac{1}{R_1^2}\|w_h\|_{L^p(B_3)}+ \frac{1}{R_1}\|w_h\|_{W^{1,p}_h(B_3)}\right)\|v_h\|_{L^{p'}(B_3)} \nonumber \\
&\lss \delta_0\|w_h\|_{\Wh(B_3)}\|v_h\|_{L^{p'}(B_3)}, \nonumber
\end{align}
where we have used Lemma \ref{lem2:55} to derive the last inequality.
Likewise, we find
\begin{align}\label{eqn3:47}
K_4 &\lss \frac{1}{R_1}\left(\sum_{e\in\E}h_e^{1-2p}\|[w_h]\|_{L^p(e_3)}^p\right)^{\frac{1}{p}}\left(\sum_{e\in\E}h_eh_e^{p'}\|\{v_h\}\|_{L^{p'}(e_3)}^{p'}\right)^{\frac{1}{p'}}\\
&\lss \frac{\delta_0}{R_1}\|w_h\|_{\Wh(B_3)}\left(h\|v_h\|_{L^{p'}(B_3)}\right) \nonumber \\
&\lss \d_0\|w_h\|_{\Wh(B_3)}\|v_h\|_{L^{p'}(B_3)} \nonumber,
\end{align}
where we have used the inequality $h\le R_1$.
Combining \eqref{eqn3:43}-\eqref{eqn3:47} we get
\begin{align}\label{eqn3:48}
I_3 \lss \frac{1}{R_1^3}\|w_h\|_{\Wh(B_3)}\|v_h\|_{W_h^{-1,p'}(B_3)} +  \d_0\|w_h\|_{\Wh(B_3)}\|v_h\|_{L^{p'}(B_3)},
\end{align}
and bringing together \eqref{eqn3:39}-\eqref{eqn3:42}, and \eqref{eqn3:48} gives us
\begin{align}\label{eqn3:49}
(\aL(\eta v_h),w_h) &\lss \frac{1}{R_1^3}\left(\|\aL v_h\|_{W_h^{-2,p'}(B_3)}+ \|v_h\|_{W_h^{-1,p'}(B_3)}  \right)\|w_h\|_{\Wh(B_3)} \\
&\quad\quad + \d_0\|w_h\|_{\Wh(B_3)}\|v_h\|_{L^{p'}(B_3)}.
\end{align}
By the definition of $\|\cdot\|_{W_h^{-2,p'}(B_3)}$ and \eqref{eqn3:49} we get
\begin{align}\label{eqn3:50}
\|\aL(\eta w_h)\|_{W_h^{-2,p'}(B_3)} &\lss \frac{1}{R_1^3}\left(\|\aL v_h\|_{W_h^{-2,p'}(B_3)}+ \|v_h\|_{W_h^{-1,p'}(B_3)}  \right) + \d_0\|v_h\|_{L^{p'}(B_3)}.
\end{align}
Using \eqref{eqn3:38} and \eqref{eqn3:50} gives us
\begin{align*}
\|v_h\|_{L^{p'}(B_1)}\lss \frac{1}{R_1^3}\left(\|\aL v_h\|_{W_h^{-2,p'}(B_3)}+ \|v_h\|_{W_h^{-1,p'}(B_3)}  \right) + \d_0\|v_h\|_{L^{p'}(B_3)}.
\end{align*}
{Since $\overline{\W}$ is compact, employing a covering argument (cf. \cite{AX:FN,Sp:DT}) then yields}
\begin{align*}
\|v_h\|_{L^{p'}(\W)}\lss \|\aL v_h\|_{W_h^{-2,p'}(\W)}+ \|v_h\|_{W_h^{-1,p'}(\W)} + \d_0\|v_h\|_{L^{p'}(\W)}.
\end{align*}
{Because $\d_0$ is small, we can absorb the last term on the right-hand side to the left-hand side
to arrive at the global estimate}
\begin{align}\label{eqn3:51}
\|v_h\|_{L^{p'}(\W)}\lss\|\aL v_h\|_{W_h^{-2,p'}(\W)}+ \|v_h\|_{W_h^{-1,p'}(\W)},
\end{align}
which is a G\"arding-type inequality.

\medskip
{\em Step 3: Duality argument on the adjoint operator.}
To control the last term in \eqref{eqn3:51} we now use a duality argument for $\aL$.  
This argument uses the regularity estimate of the original problem $\L$. 

Define the set
\begin{align*}
{X = \{g\in W^{1,p}_h(\W);\, \|g\|_{W_h^{1,p}(\W)} = 1\}.}  
\end{align*}
By the {discrete Poincar\'{e} inequality, with constant $C=C(p,\W)$, we have for all $g\in X$
\begin{align*}
\|g\|_{L^p(\W)} 
\leq   C\|g\|_{W_h^{1,p}(\W)} < \infty,
\end{align*}
{since $X$ is bounded in $W^{1,p}_h(\W)$}. Thus, $X$ is precompact} in $L^{p}(\W)$ by Sobolev embedding.  
Next we define the set
\begin{align*}
W = \{\phi := \L^{-1}g;\,g\in X\}.  
\end{align*}
Note that $\L^{-1}:L^p(\W)\to W^{2,p}(\W)\cap W_0^{1,p}(\W)\subset W^{2,p}(\Th)$ is well defined by 
well-posedness of the PDE.  Also since $\L^{-1}$ is linear and satisfies the estimate 
\begin{align*}
\|\L^{-1} g\|_{W^{2,p}_h(\W)}=\|\phi\|_{\Wh(\W)} \leq \|\phi\|_{W^{2,p}(\W)} \lss \|g\|_{L^p(\W)},
\end{align*}
it is bounded in $W^{2,p}(\Th)$.  Thus $W$ is precompact in $W^{2,p}(\Th)$.  From \cite[Lemma 5]{MC:SW}, 
for every $\tau>0$ there exists $h_*>0$ that only 
depends on $\tau$ and $\overline{W}$ such that for each $\phi\in W$ and $0<h\leq h_*$ 
there is a $\phi_h\in V_h$ such that if $k\geq 2$ we have
\begin{align}\label{eqn3:52}
\|\phi-\phi_h\|_{\Wh(\W)} \leq \tau.
\end{align}
Note by the reverse triangle inequality and \eqref{eqn3:52} we have
\begin{align*}
\|\phi_h\|_{\Wh(\W)}\leq \|\phi\|_{\Wh(\W)} \lss \|g\|_{L^p(\W)} \leq C
\end{align*}
and hence
\begin{align*}
\{ \phi_h\in V_h;\, |\phi_h - \phi|\leq \tau\}
\end{align*}
is uniformly bounded in $\phi$ and $h$. Let $g\in X$ and choose $\phi_g=\L^{-1}g\in W$ 
which tells us that $\L\phi_g=g$.  Let $v_h\in V_h$ and $\phi_h\in V_h$.  By Lemma \ref{lem3:3} 
and the definition of $\|\cdot\|_{W_h^{-2,p'}(\W)}$ we have
\begin{align*}
\int_{\W}v_h g\dx[x] &= (\L_h^{\e}\phi_g,v_h) = (\L_h^{\e}\phi_h,v_h) + (\L_h^{\e}(\phi_g-\phi_h),v_h) \\
&= (\aL v_h,\phi_h) + (\L_h^{\e}(\phi_g-\phi_h),v_h) \\
&\lss \|\aL v_h\|_{W_h^{-2,p'}(\W)}\|\phi_h\|_{\Wh(\W)} + \|\phi_g-\phi_h\|_{\Wh(\W)}\|v_h\|_{L^{p'}(\W)}.
\end{align*}
Selecting $\phi_h$ to satisfy \eqref{eqn3:52} and 
{taking the supremum on $g$ gives us}
\begin{align}\label{eqn3:53}
\|v_h\|_{W^{-1,p}(\W)} \lss \|\aL v_h\|_{W_h^{-2,p'}(\W)}\|\phi_h\|_{\Wh(\W)} + \tau\|v_h\|_{L^{p'}(\W)}.
\end{align}
Combining \eqref{eqn3:51} and \eqref{eqn3:53} yields
\begin{align}
\|v_h\|_{L^{p'}(\W)}\lss\|\aL v_h\|_{W_h^{-2,p'}(\W)}+ \tau\|v_h\|_{L^{p'}(\W)}.
\end{align}
By choosing $\tau$ sufficiently small to kick back the right-most term we have \eqref{eqn3:30.5}.
This completes the proof upon taking $h_0  = \min\{h_{\delta_0},h_*\}$.
\end{proof}

We are now ready to prove the global stability of the operator $\L_h^{\e}$.

\begin{theorem}\label{thm3:1}
Suppose that $h\leq h_0$ and $k\geq 2$.  Then there holds the following stability estimate:
\begin{align}\label{eqn3:54}
\|w_h\|_{\Wh(\W)}\lss \|\L_h^{\e} w_h\|_{L_h^p(\W)} \qquad \forall w_h\in V_h,
\end{align}
where $1<p<\infty$ if $\e = 1$, and $p=2$ if $\e\in \{0,-1\}$.
\end{theorem}
\begin{proof}
Let $w_h\in V_h$ be fixed, and consider the auxiliary problem of finding $q_h\in V_h$ such that 
\begin{align}
\label{eqn3:56}
\bigl(v_h,\aL q_h\bigr)
&=\bigl(\L_h^{\e} v_h,q_h\bigr) = \int_{\W} |D^2_h w_h|^{p-2}D^2_h w_h:D^2v_h \dx[x]\\
&\qquad\nonumber
+\sum_{e\in \EI} h_e^{1-p} \int_e { |[\nab w_h]|^{p-2} [\nab w_h]\cdot [\nab v_h]}\, \dx[S]\\
&\qquad\qquad\nonumber
+\sum_{e\in \E} h_e^{1-2p} \int_e |[w_h]|^{p-2} [w_h] [v_h]\, \dx[S] 
\qquad \forall v_h\in V_h.
\end{align}
Since $V_h$ is finite dimensional and the operator is linear, the existence is equivalent to the uniqueness.  
To show the uniqueness, let $q_h^{(1)}$ and $q_h^{(2)}$ both solve \eqref{eqn3:56}. Then by Lemma \ref{lem3:5} we get
\begin{align*}
\|q_h^{(1)}-q_h^{(2)}\|_{L^{p'}(\W)} \lss \sup_{0\neq v_h\in V_h}\frac{(\L_h^{\e} v_h,q_h^{(1)}-q_h^{(2)})}{\|v_h\|_{W^{2,p}_h(\Omega)}}
=0.
\end{align*}
Hence, \eqref{eqn3:56} has a unique solution $q_h\in V_h$. Also by Lemma  \ref{lem3:5} and H\"older's inequality,
\begin{align*}
\|q_h\|_{L^{p^\prime}(\Omega)} \lss \sup_{0\neq v_h\in V_h}\frac{(\L_h^{\e} v_h,q_h)}{\|v_h\|_{W^{2,p}_h(\W)}} \les \|w_h\|_{W^{2,p}_h(\Omega)}^{p-1}.
\end{align*}
Consequently, we find
\begin{align*}
\|w_h\|_{W^{2,p}_h(\Omega)}^p &\les (w_h,\mathcal{L}_h^* q_h) = (\mathcal{L}_h w_h,q_h)
\les \|\mathcal{L}_h w_h\|_{L^p_h(\Omega)} \|q_h\|_{L^{p'}(\Omega)} \\
&\les \|\mathcal{L}_h w_h\|_{L^p_h(\Omega)} \|w_h\|_{W^{2,p}_h(\Omega)}^{p-1}.
\end{align*}
Dividing by $\|w_h\|_{W^{2,p}_h(\Omega)}^{p-1}$ now yields the desired result.
\end{proof}

\subsection{Well-posedness and error estimates}\label{section4.3}
The goals of this subsection are to establish the well-posedness for the IP-DG scheme \eqref{eqn:dis}
and to derive the optimal order error estimates in $W^{2,p}_h$-norm for the IP-DG solutions.

\begin{theorem}
Under the assumptions of Lemma \ref{lem3:5}, the IP-DG scheme \eqref{eqn:dis} has a unique solution $u_h\in V_h$ 
{such that}
\begin{align}\label{eqn3:55}
\|u_h\|_{\Wh(\W)}\lss \|f\|_{L^p(\W)}.
\end{align}
\end{theorem}
\begin{proof}
Since \eqref{eqn:dis} is equivalent to a linear system, hence it suffices to prove the uniqueness. 
To show the uniqueness, we first prove \eqref{eqn3:55}.

Let $u_h\in V_h$ be a solution of \eqref{eqn:dis}, then from \eqref{eqn3:54} and the definition 
of $\|\cdot\|_{L_h^p(\W)}$  we have
\begin{align*}
\|u_h\|_{\Wh(\W)}\lss \|\L_h^{\e} u_h\|_{L_h^p(\W)} 
= \sup_{v_h\in V_h}\frac{(\L_h^{\e} u_h,v_h)}{\|v_h\|_{L^{p'}(\W)}} 
= \sup_{v_h\in V_h}\frac{(f,v_h)}{\|v_h\|_{L^{p'}(\W)}} \leq \|f\|_{L^p(\W)}.
\end{align*}
Hence, \eqref{eqn3:55} holds.

Suppose that $u_h^1,u_h^2\in V_h$ solve \eqref{eqn:dis}. Let $\tilde{u}_h = u_h^1-u_h^2$.  
Then by \eqref{eqn3:55} we have
\begin{align*}
\|\tilde{u}_h \|_{\Wh(\W)} \leq \|0\|_{L^p(\W)} = 0.
\end{align*}
Since $\tilde{u}_h \in V_h$ with $\|\tilde{u}_h \|_{\Wh(\W)}=0$ we conclude that
$\tilde{u}_h\in C^{1}(\W)$, $\tilde{u}_h\big|_{\partial\W}=0$, and $D_h^2\tilde{u}_h =0 $ in $\W$.  
The only way this can happen is if $\tilde{u}_h =0$.  Thus, the IP-DG solution must be unique.
The proof is complete.
\end{proof}

Next we show a C\'ea-type lemma for the IP-DG scheme, which immediately deduces the
optimal order error estimates in the $\Wh$-norm.

\begin{theorem}
Suppose that $h\leq h_0$ and $k\geq 2$.  Let $u\in W^{2,p}\cap W_0^{1,p}(\W)$ be the solution
of problem \eqref{NonDivProblem} and $u_h\in V_h$ solve \eqref{eqn:dis}.  Then
\begin{align}\label{eqn3:58}
\|u-u_h\|_{\Wh(\W)}\lss \inf_{w_h\in V_h}\|u-w_h\|_{\Wh(\W)}.
\end{align}
Moreover, if $u\in W^{s,p}(\W)$ for some $s\geq 2$, we have
\begin{align}\label{eqn3:59}
\|u-u_h\|_{\Wh(\W)}\lss h^{r-2}\|u\|_{W^{r,p}(\W)}, \qquad r = \min\{s,k+1\}.
\end{align}
\end{theorem}
\begin{proof}
By the consistency of $\L_h^{\e}$ we have the following Galerkin orthogonality:
\begin{align}\label{eqn3:60}
\bigl(\L_h^{\e}(u-u_h),v_h \bigr) = 0 \qquad\forall v_h\in V_h.
\end{align}
Let $w_h\in V_h$, by Theorem \ref{thm3:1}, Lemma \ref{lem3:3}, \eqref{eqn3:60}, and the definition 
of $\|\cdot\|_{L_h^p(\W)}$ we have
\begin{align}\label{eqn3:61}
\|u_h-w_h\|_{\Wh(\W)} &\lss \|\L_h^{\e}(u_h-w_h)\|_{L_h^p(\W)} 
= \sup_{0\neq v_h\in V_h} \frac{(\L_h^{\e}(u_h-w_h),v_h)}{\|v_h\|_{L^{p'}(\W)}} \\
&= \sup_{0\neq v_h\in V_h} \frac{(\L_h^{\e}(u-w_h),v_h)}{\|v_h\|_{L^{p'}(\W)}} 
= \|\L_h^{\e}(u-w_h)\|_{L_h^p(\W)} \nonumber\\
&\lss \|u-w_h\|_{\Wh(\W)}. \nonumber
\end{align}
Thus by \eqref{eqn3:61} and the triangle inequality we get
\begin{align}
\|u-u_h\|_{\Wh(\W)} \leq \|u-w_h\|_{\Wh(\W)} + \|u_h-w_h\|_{\Wh(\W)} \lss \|u-w_h\|_{\Wh(\W)}.
\end{align}
Taking the infimum on both sides over all $w_h\in V_h$ yields \eqref{eqn3:58}. Finally, \eqref{eqn3:59} 
follows from taking $w_h = I_hu$ and using the finite element interpolation theory \cite{Sp:BS}. 
The proof is complete.
\end{proof}

\section{Numerical Experiments} \label{section5}

In this section we present a number of $2$-D numerical tests to verify our error estimate and 
to gauge the performance of our IP-DG methods. In particular, we shall compare our IP-DG methods to the related
conforming finite element counterpart developed in \cite{AX:FN}.  Moreover,  we shall also perform 
numerical tests which are not covered by our convergence theory, this includes the cases when the coefficient 
matrix is either discontinuous or degenerate.

\subsection{H\"{o}lder continuous coefficient}

For this test we take $A$ as the following H\"{o}lder continuous matrix-valued function: 
\[
A(x) = \left[
\begin{array}{c c}
|x|^{1/2}+1 & -|x|^{1/2} \\
-|x|^{1/2} & 5|x|^{1/2}+1
\end{array}
\right], \qquad x\in\R^2.
\]
Let $\W=(-1/2,1/2)^2$ and choose $f$ such that the exact solution is given by
\begin{align*}
u(x_1,x_2) = \sin(2\pi x_1)\sin(2\pi x_2)\exp(x_1\cos(x_2)),
\end{align*}
which has zero trace on the boundary. 

Figure \ref{figureholder} shows the errors in the $L^2(\W), W_h^{1,2}(\W),$ and $W_h^{2,2}(\W)$ norms 
of both the symmetrically and incompletely induced methods. The convergence rates observed for the symmetrically 
induced method are
\begin{align*}
\|u-u_h\|_{L^2(\W)} &= \mathcal{O}(h^{k+1}) \text{ for all } k,\\
\|\grad_h(u-u_h)\|_{L^2(\W)} &= \mathcal{O}(h^{k}) \text{ for all } k, \\
\|D_h^2(u-u_h)\|_{L^2(\W)} &= \mathcal{O}(h^{k-1}) \text{ for } k = 2,3.
\end{align*}
As expected, these convergence rates are optimal. However, for the incompletely induced method we find that 
the rate of convergence in the $L^2$-norm is sub-optimal for even degree polynomials and optimal with all other 
norms and degrees.  This should be expected since the incomplete scheme is sub-optimal even for smooth $A$ \cite{F:BR}.

\begin{figure}[htb]
\label{figureholder}
\includegraphics[width=0.45\textwidth]{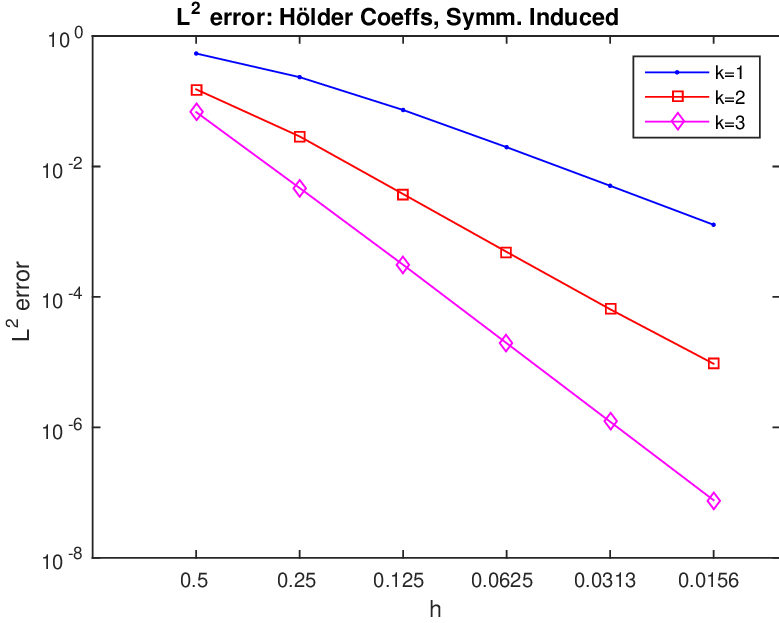}
\includegraphics[width=0.45\textwidth]{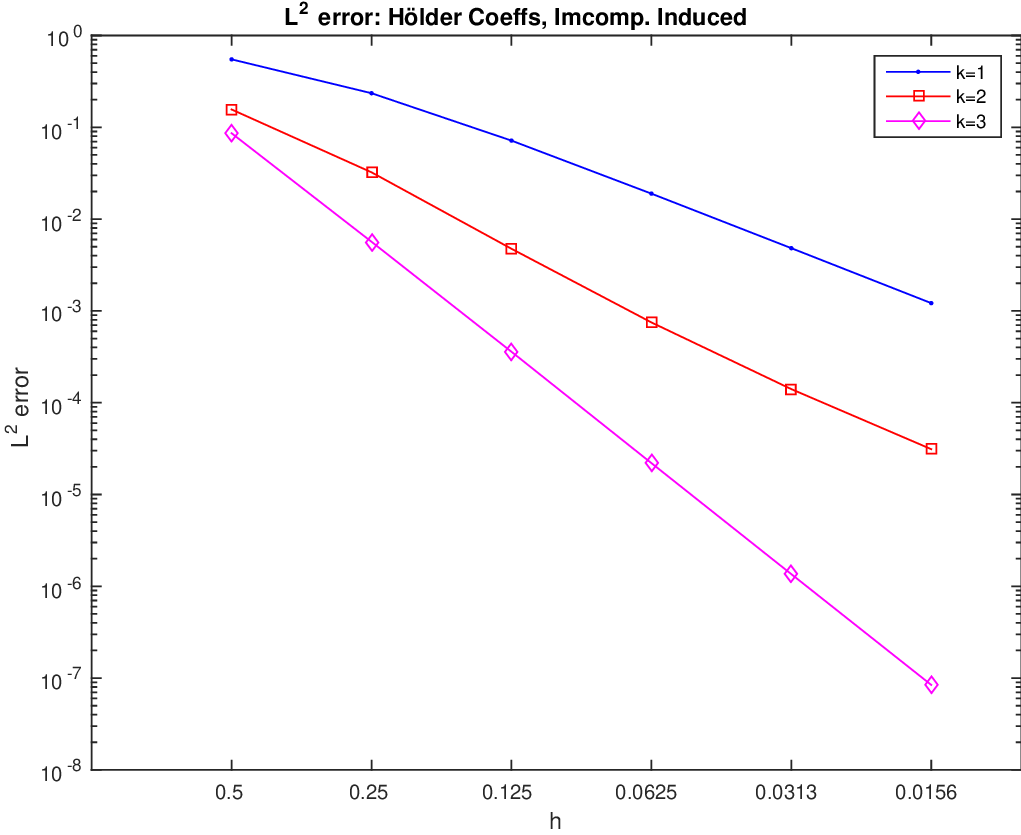}

\includegraphics[width=0.45\textwidth]{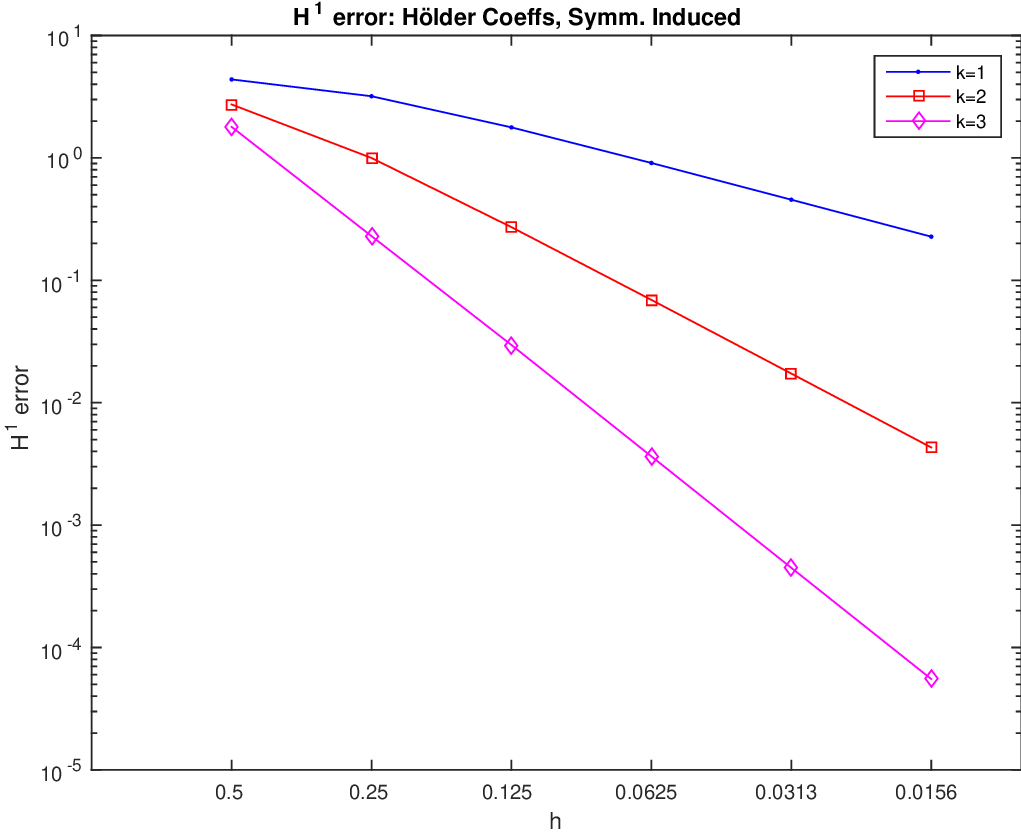}
\includegraphics[width=0.45\textwidth]{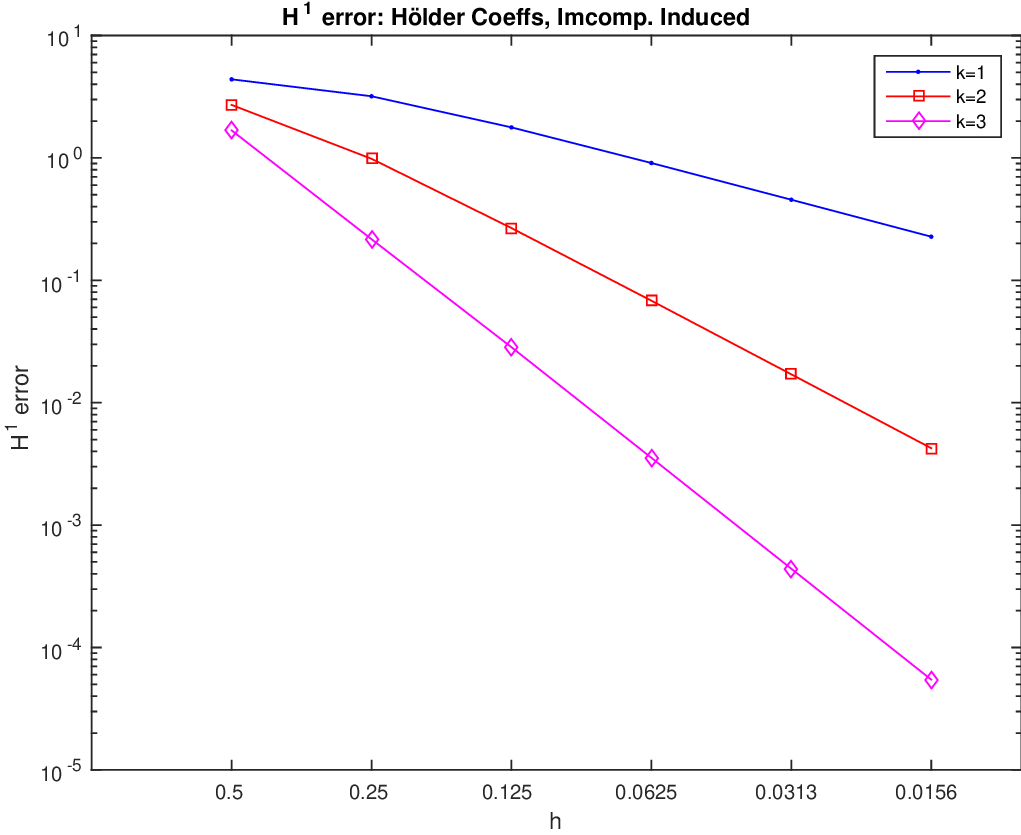}

\includegraphics[width=0.45\textwidth]{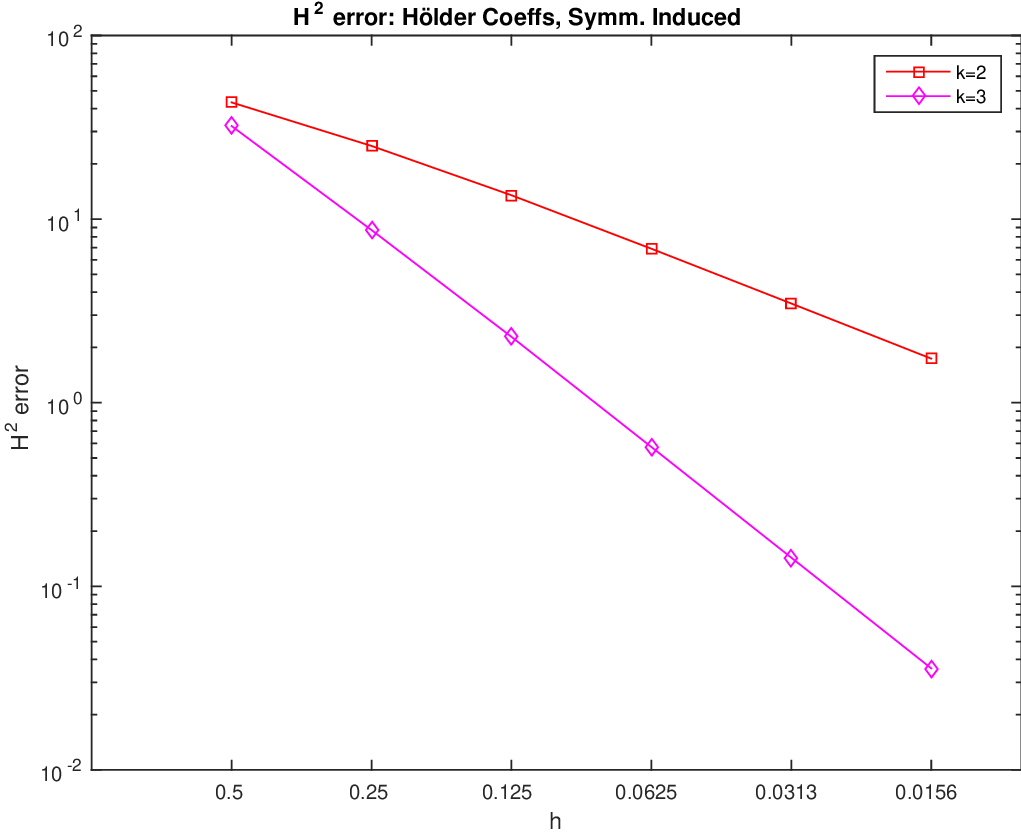}
\includegraphics[width=0.45\textwidth]{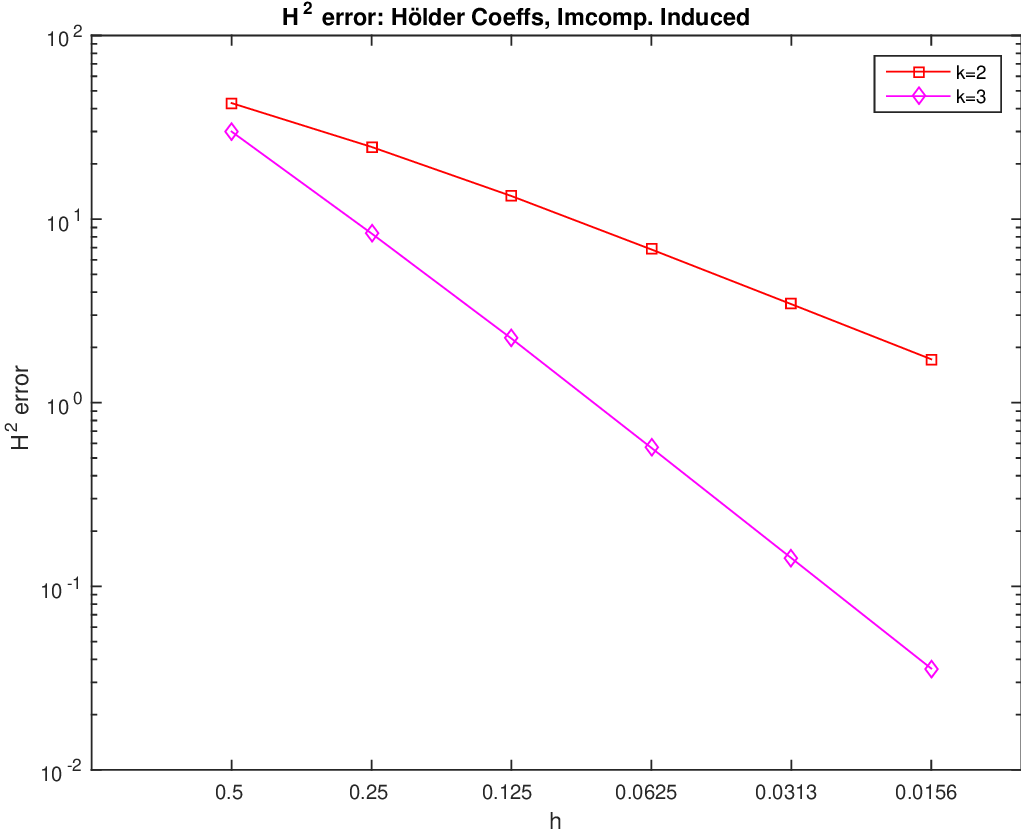}
\caption{The $L^2$ (top), piecewise $H^1$ (middle), and piecewise $H^2$ (bottom) errors 
for both the symmetrically (left) and incompletely (right) induced schemes with polynomial 
degree $k=1,2,3$.  $\gamma_e\equiv 100$ is used as the penalty parameter.} 
\end{figure}
 
\subsection{Uniformly continuous coefficients}

In this test we take $\W=(0,1/2)^2$ and let
\[
A(x) = \left[
\begin{array}{c c}
-\dfrac{5}{\log(|x|)}+15 & 1 \\
1 & -\dfrac{1}{\log(|x|)}+3
\end{array}
\right].
\]
$f$ is chosen such that $u(x) = |x|^{7/4}$ is the exact solution.  From \cite{AX:FN} we see that the expected 
convergence rates are 
\begin{align*}
\|\grad_h(u-u_h)\|_{L^2(\W)} &= \mathcal{O}(h^{\min\{k,7/4-\d\}}) \text{ for all } k, \\
\|D_h^2(u-u_h)\|_{L^2(\W)} &= \mathcal{O}(h^{\min\{k,7/4-\d\}-1}) \text{ for } k = 2,3
\end{align*}
for any $\d>0$. 

Figure \ref{figureunif} gives the computed results for both the symmetrically and incompletely induced schemes 
which match exactly the expected rates of convergence.

\begin{figure}[htb] \label{figureunif}
\includegraphics[width=0.45\textwidth]{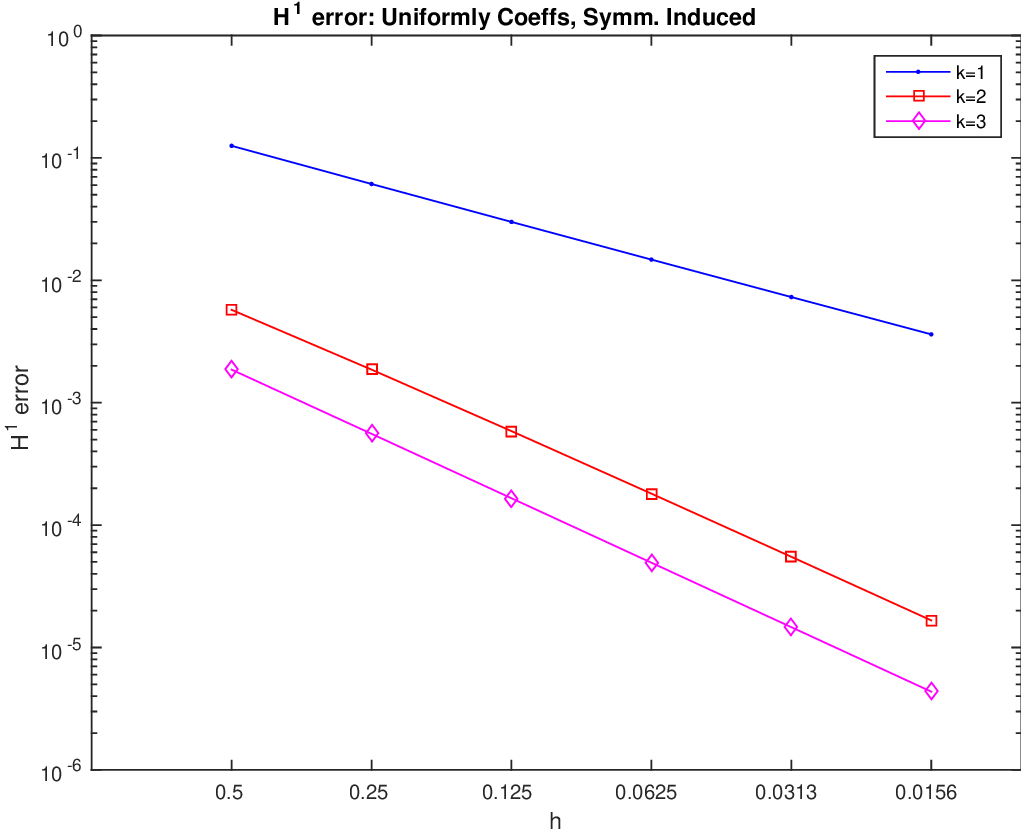}
\includegraphics[width=0.45\textwidth]{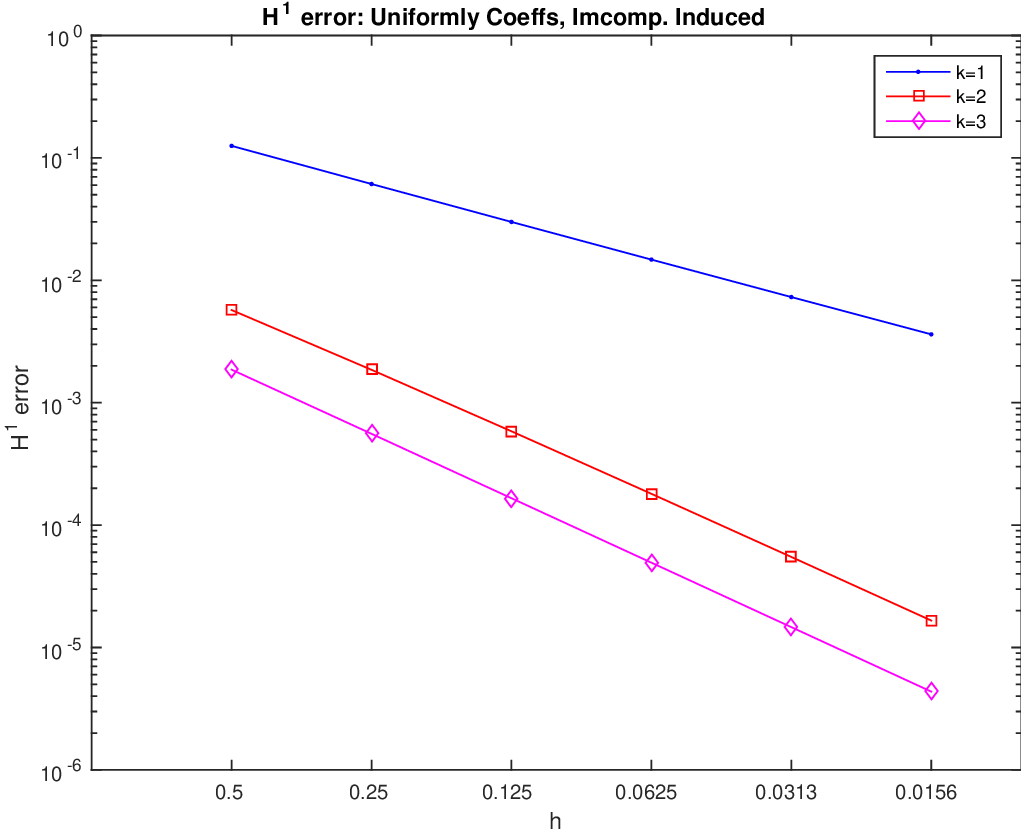}

\includegraphics[width=0.45\textwidth]{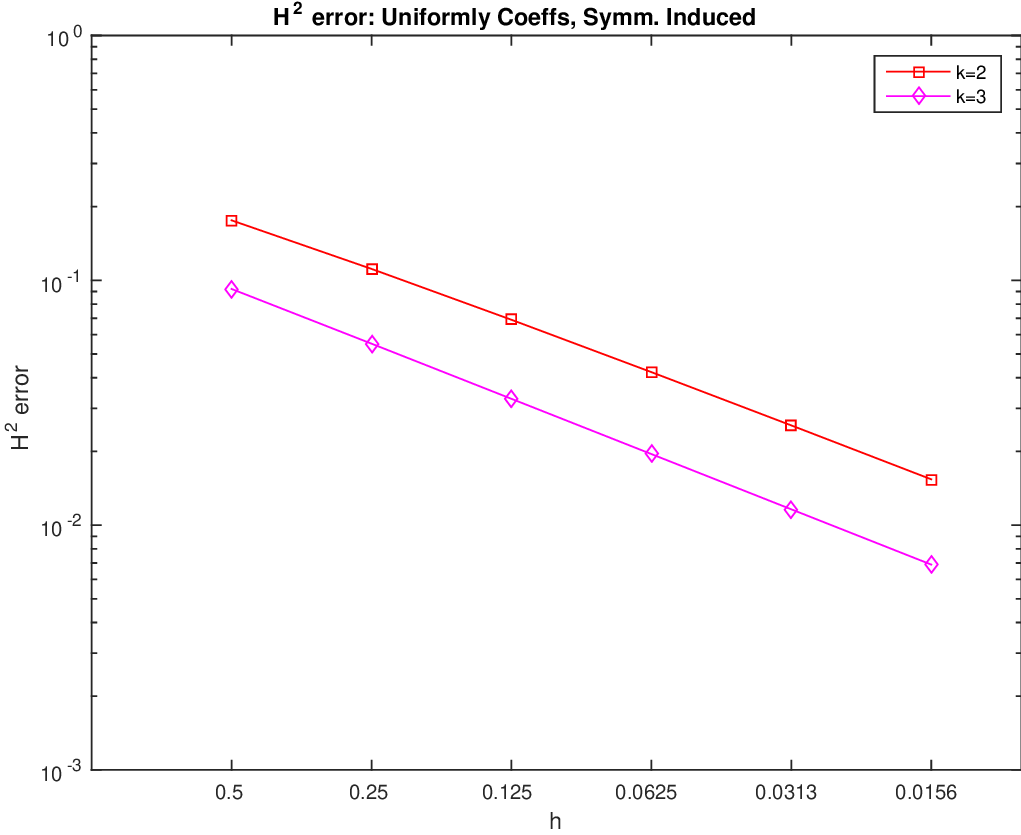}
\includegraphics[width=0.45\textwidth]{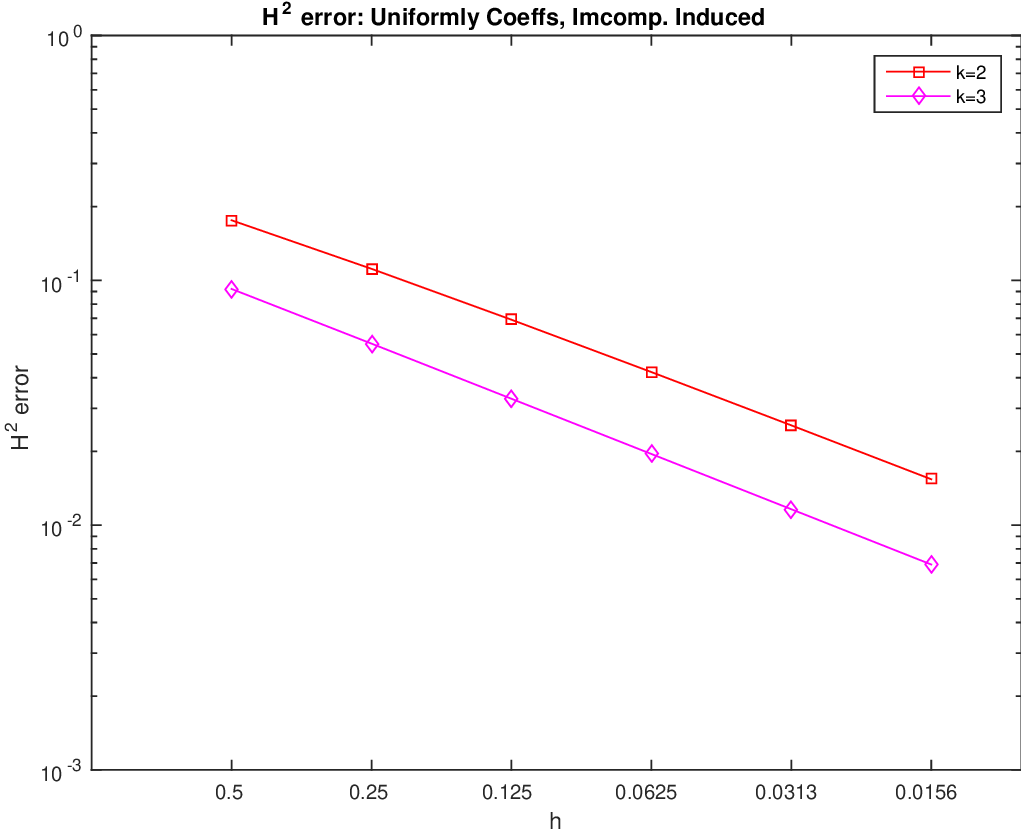}
\caption{The piecewise $H^1$ (top) and piecewise $H^2$ (bottom) errors for both the symmetrically (left) 
and incompletely (right) induced schemes with polynomial degree $k=1,2,3$. $\gamma_e\equiv 1000$ is used as
the penalty parameter.}
\end{figure}

\subsection{Degenerate coefficients}

In this test we take $\W=(0,1)^2$ and the matrix
\[
A(x) = \frac{16}{9}\left[
\begin{array}{c c}
x_1^{2/3} & -x_1^{1/3}x_2^{1/3} \\
-x_1^{1/3}x_2^{1/3} & x_2^{2/3}
\end{array}
\right].
\]
$f=0$ and the exact solution $u(x) = x_1^{4/3}-x_2^{4/3}$.  For an explanation for this example we refer to \cite{AX:FN}.  
Note that $\det(A)=0$ for every $x\in\W$ so this PDE is degenerate everywhere and is outside the case considered in this
paper. We also observe that $u\in W^{m,p}(\W)$ provided $(4-3m)p>-1$. 

Figure \ref{figuredegenerate} shows the $L^2$ and piecewise $H^1$ errors for both the symmetrically and incompletely 
induced methods. The numerical results suggest the following rates of convergence:
\begin{align*}
 \|u-u_h\|_{L^2(\W)} &= \mathcal{O}(h^{4/3}),\\
 \|\grad_h(u-u_h)\|_{L^2(\W)} &= \mathcal{O}(h^{5/6}) 
\end{align*}
for $k=1,2,3$. These rates are consistent with the results of the related conforming finite element 
method given in \cite{AX:FN}.

\begin{figure}[htb]
\label{figuredegenerate}
\includegraphics[width=0.45\textwidth]{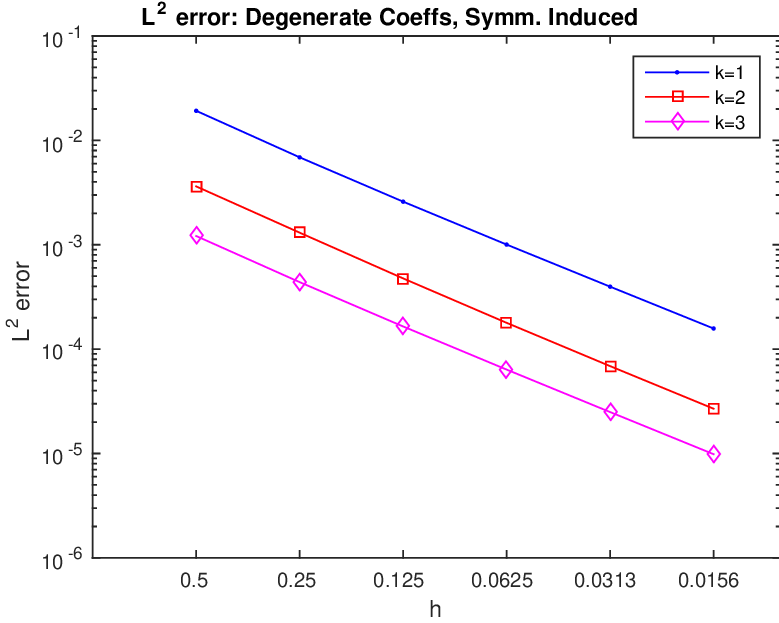}
\includegraphics[width=0.45\textwidth]{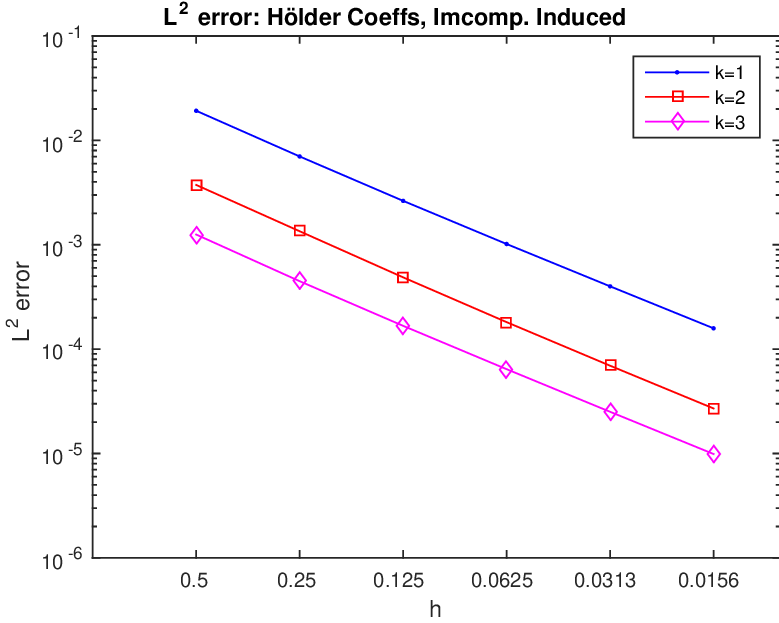}

\includegraphics[width=0.45\textwidth]{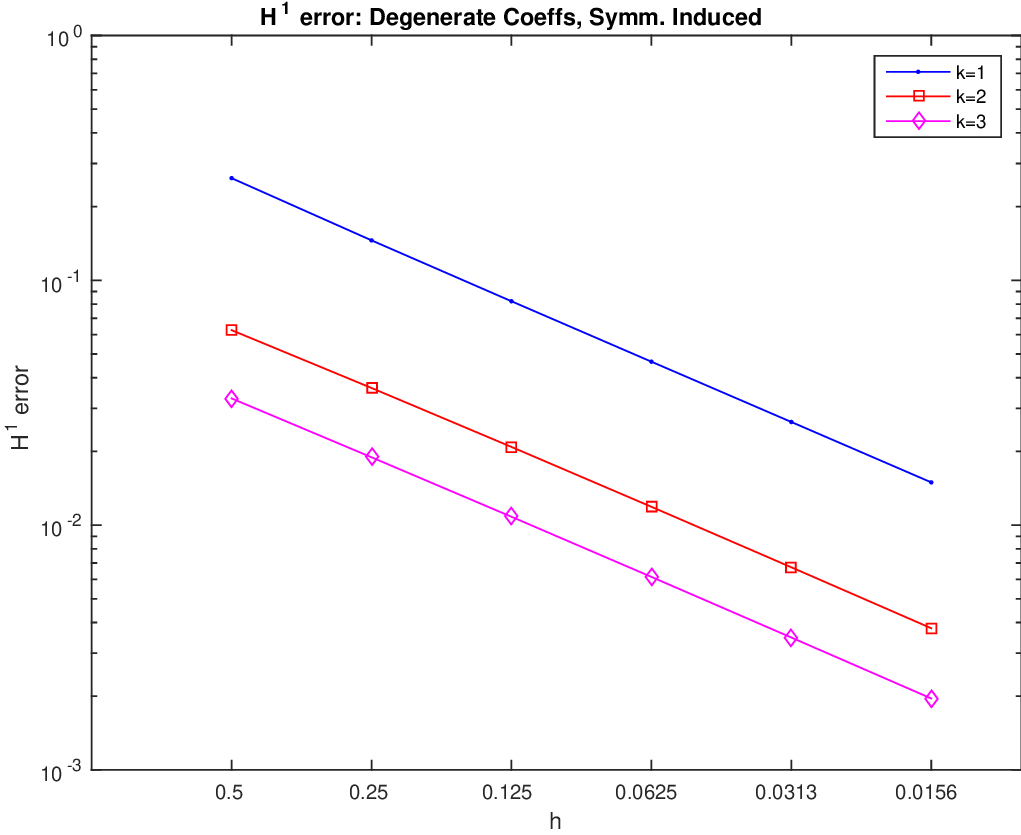}
\includegraphics[width=0.45\textwidth]{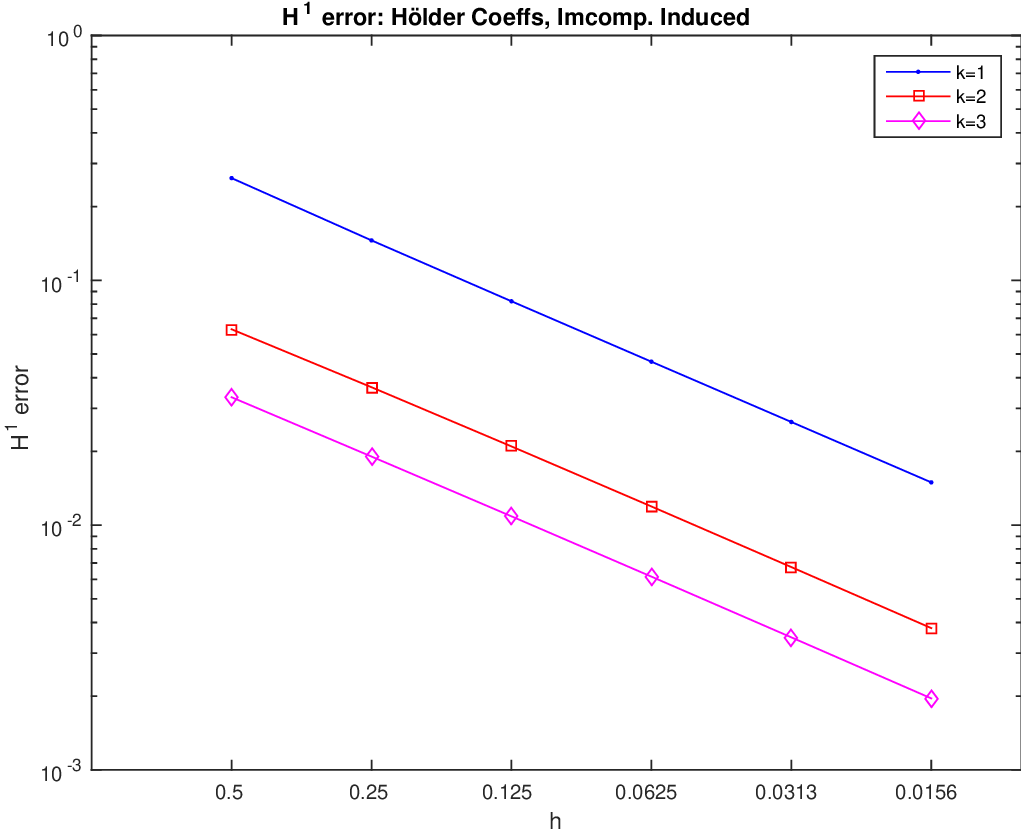}
\caption{The $L^2$ (top) and piecewise $H^1$ (bottom) errors for both the symmetrically (left) and incompletely (right) 
induced schemes with polynomial degree $k=1,2,3$. $\gamma_e\equiv 100$ is used as the penalty parameter. }
\end{figure}

\subsection{$L^{\infty}$ Cord\`{e}s coefficients}

Our next test is taken from \cite{SmearsSuli,AX:WW} where a different DG method and a weak Galerkin method
were used to solve this problem.  Let $\W=[-1,1]^2$ and 
\[
A(x) = \frac{16}{9}\left[
\begin{array}{c c}
2 & x_1 x_2/|x_1 x_2| \\
x_1 x_2/|x_1 x_2| & 2
\end{array}
\right].
\]
$f$ is chosen so that the exact solution is $u(x) = x_1 x_2 \bigl(1-e^{1-|x_1|}\bigr)\bigl(1-e^{1-|x_2|}\bigr)$. 
Notice that the matrix $A$ is discontinuous across the $x_1$-axis and $x_2$-axis, and it satisfies the Cord\`{e}s 
condition. While our convergence theory does not apply to this example, we still compute the numerical solution on 
a uniform triangulation that has edges on all discontinuities of $A$. Due to its inconsistent behavior we list
the $L^2$ error and convergence rates in Table \ref{tab:cordes}.  The following $H^1$ semi-norm rates are observed:
\begin{align*}
 \|\grad_h (u-u_h)\|_{L^2(\W)} = \mathcal{O}(h^{k})
\end{align*}
for $k=1,2,3$ as shown in Figure \ref{figurecordes}.

\begin{table}[htb]
\begin{center}
\caption{The $L^2$ errors and rates for the symmetrically induced method.  The rates for the incompletely induced 
method are similar.  $\gamma_e\equiv 10000$ is used as the penalty parameter.}
\begin{tabular}{ |c|c|c|c|c|c|c|}\cline{2-7}
\multicolumn{1}{ c }{}      & \multicolumn{2}{ | c | }{$k=1$} & \multicolumn{2}{  c }{$k=2$} & \multicolumn{2}{ |  c  |}{$k=3$} \\ \hline
$h$    & $\|u-u_h\|_{L^2(\W)}$ & rate & $\|u-u_h\|_{L^2(\W)}$ & rate & $\|u-u_h\|_{L^2(\W)}$ & rate \\ \hline
1      & 1.3e-1 &  -   & 7.7e-2 &  -   & 2.6e-2 &  -   \\ \hline 
1/2    & 8.9e-2 & 0.58 & 1.8e-2 & 2.09 & 1.5e-3 & 4.12 \\ \hline 
1/4    & 4.6e-2 & 0.95 & 2.9e-3 & 2.62 & 7.6e-4 & 4.27 \\ \hline
1/8    & 1.9e-2 & 1.22 & 4.8e-2 & 2.62 & 4.2e-6 & 4.19 \\ \hline
1/16   & 7.6e-3 & 1.35 & 8.0e-5 & 2.57 & 3.3e-7 & 3.65 \\ \hline
1/32   & 2.9e-3 & 1.41 & 1.4e-5 & 2.54 & 3.2e-8 & 3.36 \\ \hline
\end{tabular} \label{tab:cordes}
\end{center}
\end{table}

\begin{figure}[htb]
\label{figurecordes}
\includegraphics[width=0.45\textwidth]{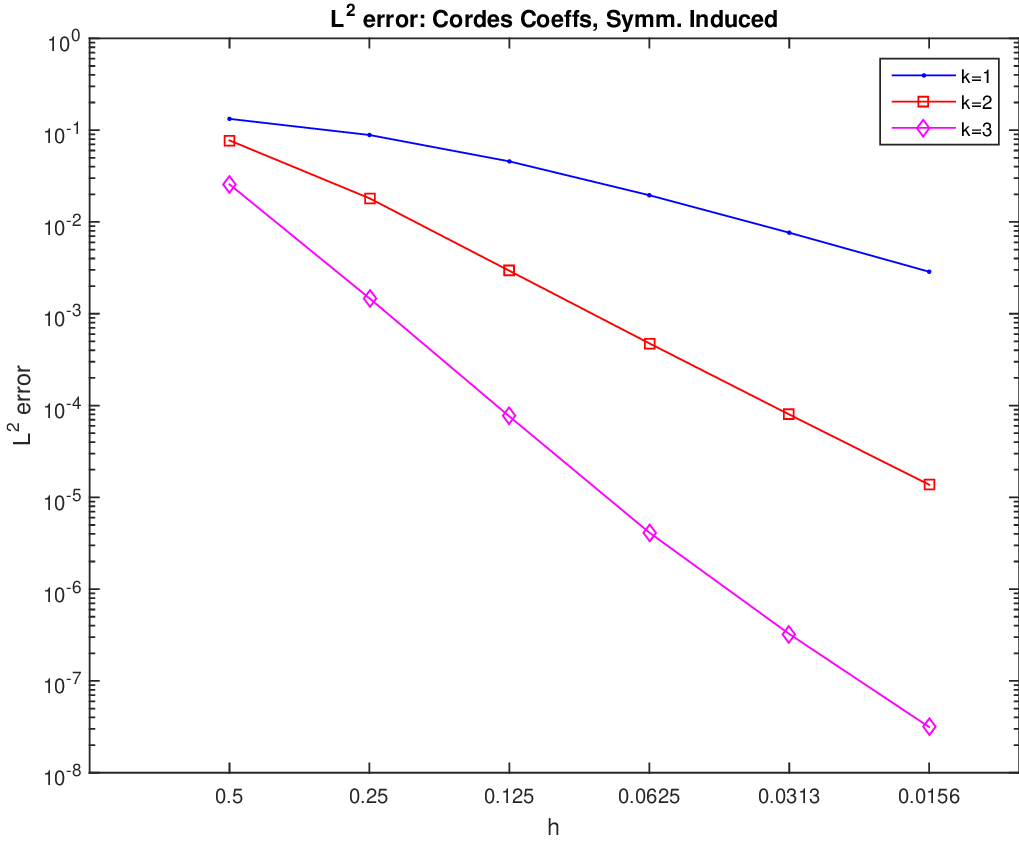}
\includegraphics[width=0.45\textwidth]{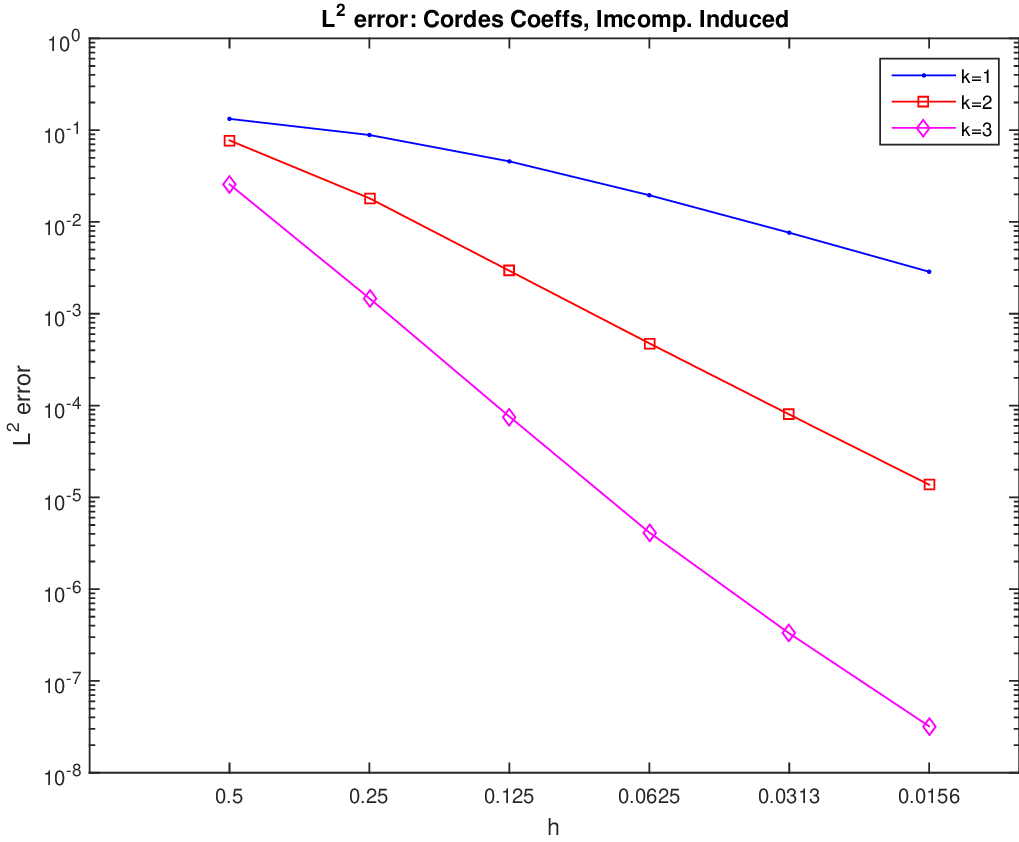}

\includegraphics[width=0.45\textwidth]{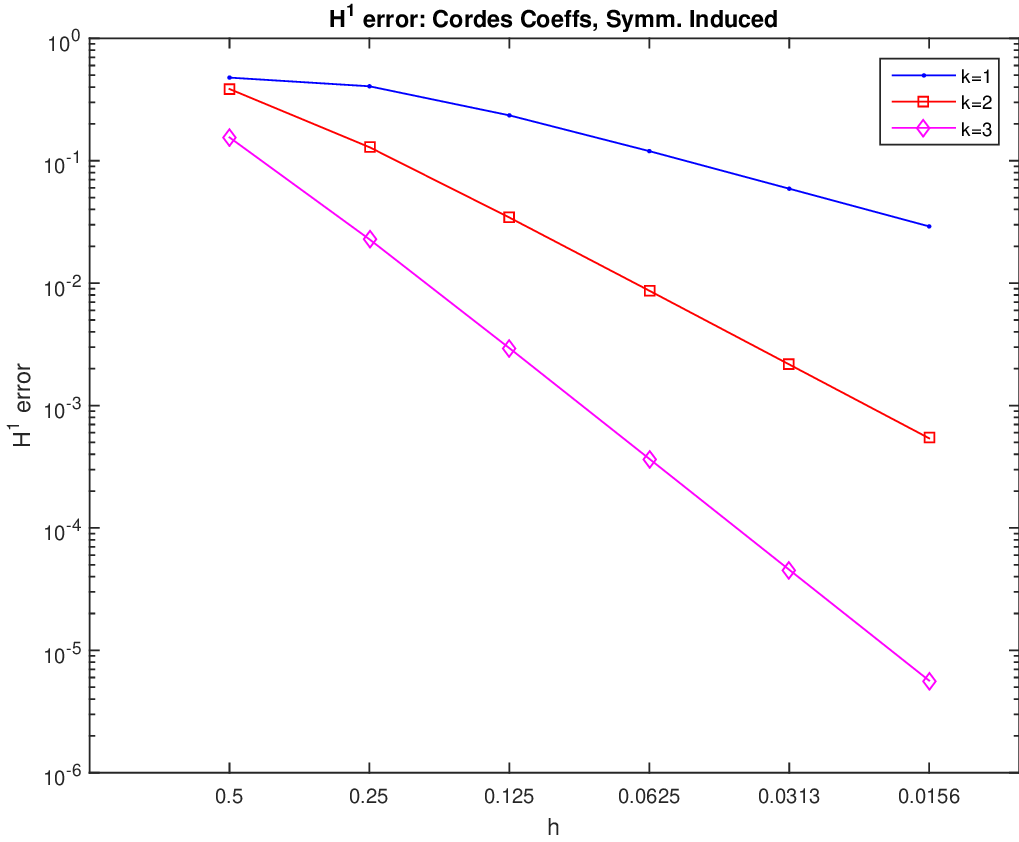}
\includegraphics[width=0.45\textwidth]{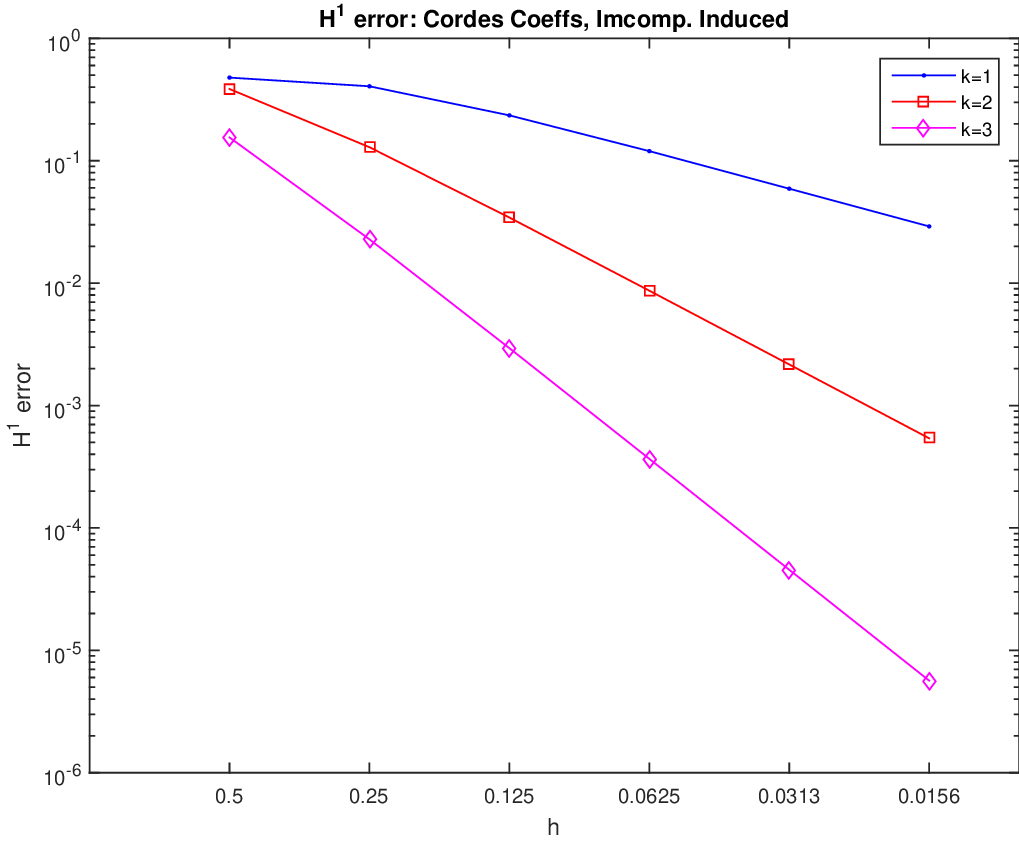}
\caption{The $L^2$ (top) and piecewise $H^1$ (bottom) errors for both the symmetrically (left) and incompletely (right) 
induced schemes with polynomial degree $k=1,2,3$. $\gamma_e\equiv 10000$ is used as the penalty parameter. }
\end{figure}


\begin{thebibliography}{99}
\bibliographystyle{abbrv}

\bibitem{NM:AH}
L. Angermann and C. Henke.
\newblock Interpolation, projection and hierarchical bases in
  discontinuous Galerkin methods.
\newblock {\em Numer. Math. Theory Methods Appl.}, 8(3):425--450, 2015.

\bibitem{Sp:Br}
S. Bernstein.
\newblock Sur la g{\'e}n{\'e}ralisation du probl{\`e}me de dirichlet.
\newblock {\em Mathematische Annalen}, 62(2):253--271.

\bibitem{Sp:BS}
S.~C. Brenner and L.~R. Scott.
\newblock {\em The Mathematical Theory of Finite Element Methods}, volume~15 of
  {\em Texts in Applied Mathematics}.
\newblock Springer, New York, third edition, 2008.

\bibitem{BrezisBook}
H.~Brezis.
{\em Functional analysis, Sobolev spaces and partial differential equations},
Universitext, Springer, New York, 2011.


\bibitem{AM:CG}
{\sc L. A.~Caffarelli and C. E.~Guti\'errez}.
{\em Properties of the solutions of the linearized Monge-Amp\`ere equation},
Amer. J. Math., 119(2):423--465, 1997.

\bibitem{ChenChen04}
Z.~Chen and H.~Chen.
{\em Pointwise error estimates of discontinuous Galerkin methods with penalty for second-order elliptic problems},
SIAM J. Numer. Anal., 42(3):1146--1166, 2004.

\bibitem{SIAM:Ciarlet}
P. Ciarlet.
\newblock {\em The Finite Element Methods for Elliptic Problems},
\newblock Society for Industrial and Applied Mathematics (SIAM), 2002.

\bibitem{Crandall_Ishii_Lions}
M. G.~Crandall, H.~Ishii, and P.-L.~Lions.
\newblock {\em User's guide to viscosity solutions of second order partial differential equations.}
\newblock {\em Bull.~Amer.~Math.~Soc.}, 27(1):1--67, 1992.

\bibitem{DednerPryer13}
A.~Dedner and T.~Pryer.
{\em Discontinuous Galerkin methods for non variational problems},
 arXiv:1304.2265v1 [math.NA].
 
 \bibitem{DouglasDupont79}
J.~Douglas, T.~Dupont, P.~Percell, and R.~Scott.
{\em A family of $C^1$ finite elements with optimal approximation properties for 
various Galerkin methods for 2nd and 4th order problems},
RAIRO Anal. Num\'er., 13(3):226--255, 1979.


\bibitem{AX:FN}
X.~Feng, L.~Hennings, and M.~Neilan.
\newblock Finite element methods for second order linear elliptic partial differential equations 
in non divergence form, Math. Comp., to appear.


\bibitem{SP:Flem}
W.~H. Fleming and H.~M. Soner.
\newblock {\em Controlled Markov Processes and Viscosity Solutions (Stochastic
  Modeling and Applied Probability)}.
\newblock Springer, 2010.

\bibitem{Sp:DT}
D. Gilbarg and N.~S. Trudinger.
\newblock {\em Elliptic Partial Differential Equations of Second Order}.
\newblock Classics in Mathematics. Springer-Verlag, Berlin, 2001.

\bibitem{Houston11}
E.H.~Georgoulis, P.~Houston, and J.~Virtanen.
{\em An a posteriori error indicator for discontinuous Galerkin approximations of fourth-order elliptic problems},
IMA J. Numer. Anal., 31(1):281--298, 2011.




%


\bibitem{LakkisPryer11}
O.~Lakkis, and T.~Pryer.
{\em A finite element method for second order nonvariational elliptic problems},
SIAM J. Sci. Comput., 33(2):786--801, 2011.

\bibitem{MaugeriBook}
A.~Maugeri, D.~K.~Palagachev, and L.~Softova.
{\em Elliptic and parabolic equations with discontinuous coefficients.}
Mathematical Research, 109., Wiley-VCH Verlag Berlin GmbH, Berlin, 2000.

\bibitem{Nadirashvili97}
N.~Nadirashvili. 
{\em Nonuniqueness in the martingale problem and the Dirichlet problem for uniformly elliptic operators},
 Ann. Scuola Norm. Sup. Pisa Cl. Sci., 24(3):537--549, 1997.

\bibitem{Neilan13}
M.~Neilan.
{\em Quadratic finite element approximations of the Monge-Amp\`ere equation},
J. Sci. Comput., 54(1):200--226, 2013.

\bibitem{NitscheSchatz74}
J.A.~Nitsche and A.H.~Schatz.
{\em Interior estimates for Ritz-Galerkin methods},
Math. Comp., 28:937--958, 1974.

\bibitem{NochettoZhang16}
R.H.~Nochetto and W.~Zhang.
{\em Discrete ABP estimate and convergence rates for linear elliptic
equations in non--divergence form}, arXiv:1411.6036 [math.NA].

\bibitem{SIAM:BCO}
J. E.~Osborn, I.~Babuska, and G.~Caloz.
\newblock {\em Special finite element methods for a class of second order elliptic
  problems with rough coefficients}.
\newblock { SIAM J. Numer. Anal.}, 31(4):945--981, 1994.



\bibitem{F:BR}
B. Rivi{\`e}re.
\newblock {\em Discontinuous Galerkin Methods for Solving Elliptic and
  Parabolic Equations: Theory and implementation}, volume~35 of {\em Frontiers in Applied Mathematics}.
\newblock Society for Industrial and Applied Mathematics (SIAM), Philadelphia,
  PA, 2008.
  
\bibitem{Safonov99}
M. V. Safonov. 
{\em Nonuniqueness for second-order elliptic equations with measurable coefficients},
SIAM J. Math. Anal., 30(4):879--895, 1999.

\bibitem{MC:SW}
A.~H. Schatz and J. Wang.
\newblock Some new error estimates for Ritz-Galerkin methods with minimal
  regularity assumptions.
\newblock {\em Math. Comp.}, 65(213):19--27, 1996.

\bibitem{Schatz98}
A.~H. Schatz.
{\em Pointwise error estimates and asymptotic error expansion inequalities for the finite element 
method on irregular grids. I. Global estimates}, Math. Comp., 67(223):877--899, 1998.

\bibitem{SmearsSuli}
I.~Smears and E.~S\"uli.
{\em Discontinuous Galerkin finite element approximation of
non-divergence form elliptic equations with Cord\`es coefficients},
SIAM J. Numer. Anal., 51(4):2088--2106, 2013.

\bibitem{SmearsSuli14}
I.~Smears and E.~S\"uli.
{\em Discontinuous Galerkin finite element approximation of Hamilton-Jacobi-Bellman equations 
with Cordes coefficients}, SIAM J. Numer. Anal., 52(2):993--1016, 2014.


\bibitem{AX:WW}
C.~Wang and J.~Wang.
\newblock {A primal-dual weak Galerkin finite element method for second order
 elliptic equations in non-divergence form}.
\newblock {\em 	arXiv:1510.03499 [math.NA]}.

\end{thebibliography}

\appendix
\section{Proof of Lemma \ref{lem2:4}}
\begin{proof}
From \cite[Lemma 3]{NM:AH} we have the following estimates for $\tI$ 
\begin{align}\label{eqn2:40}
h^{mp}|\eta v_h-\tI(\eta v_h)|_{W^{m,p}(T)}^p\lss h^{p(k+1)}|\eta v_h|_{W^{k+1,p}(T)}^p,
\quad 0\leq m\leq k+1.
\end{align}
By the assumptions on $\eta$, the fact that $|v_h|_{W^{k+1,p}(T)}=0$, and a standard inverse inequality we get
\begin{align}\label{eqn2:41}
|\eta v_h|_{W^{k+1,p}(T)}&\lss\sum_{|\a|+|\beta|=k+1}\int_{T}|D^{\a}\eta|^p|D^{\beta}v_h|^p\dx[x]\\
&\lss \sum_{j=0}^k\frac{1}{d^{p(k+1-j}}|v_h|_{W^{j,p}(T)}^p\lss \sum_{j=0}^k \frac{h^{-jp}}{d^{p(k+1-j)}}\|v_h\|_{L^p(T)}^p.
\nonumber
\end{align}
It follows from \eqref{eqn2:40} and \eqref{eqn2:41} with $h\leq d$ that
\begin{align*}
h^{mp}|\eta v_h-\tI(\eta v_h)|_{W^{m,p}(T)}^p
\lss \sum_{j=0}^k\frac{h^{p(k+1-j)}}{d^{p(k+1-j)}}\|v_h\|_{L^p(T)}^p\lss \frac{h^p}{d^p}\|v_h\|_{L^p(T)}.
\end{align*}
Thus we have
\begin{align*}
\|\eta v_h-\tI(\eta v_h)\|_{L^p(D)} 
&\lss \sum_{\substack{T\in\Th \\ T\cap D\neq\emptyset}}|\eta v_h-\tI(\eta v_h)|_{L^p(T)}^p\\
&\lss\sum_{\substack{T\in\Th \\ T\cap D\neq\emptyset}}\frac{h^p}{d^p}\|v_h\|_{L^p(T)}
\lss\frac{h^p}{d^p}\|v_h\|_{L^p(D_h)},
\end{align*}
\begin{align*}
h^p\|\grad_h(\eta v_h-\tI(\eta v_h))\|_{L^p(D)} 
&\lss \sum_{\substack{T\in\Th \\ T\cap D\neq\emptyset}}h^p|\eta v_h-\tI(\eta v_h)|_{W^{1,p}(T)}^p\\
&\lss\sum_{\substack{T\in\Th \\ T\cap D\neq\emptyset}}\frac{h^p}{d^p}\|v_h\|_{L^p(T)}
\lss\frac{h^p}{d^p}\|v_h\|_{L^p(D_h)}, \\
h^{2p}\|\grad_h(\eta v_h-\tI(\eta v_h))\|_{L^p(D)} 
&\lss \sum_{\substack{T\in\Th \\ T\cap D\neq\emptyset}}h^{2p}|\eta v_h-\tI(\eta v_h)|_{W^{2,p}(T)}^p\\
&\lss\sum_{\substack{T\in\Th \\ T\cap D\neq\emptyset}}\frac{h^p}{d^p}\|v_h\|_{L^p(T)}
\lss\frac{h^p}{d^p}\|v_h\|_{L^p(D_h)}.
\end{align*}
Hence \eqref{eqn2:37}, \eqref{eqn2:38}, and \eqref{eqn2:38.5} hold. 

To show \eqref{eqn2:39},  using \eqref{eqn2:41} and an inverse estimate we have
\begin{align}\label{eqn2:42}
h^{p(k-1)}|\eta v_h|_{W^{k+1,p}(T)}^p&\lss \sum_{j=0}^{k} \frac{h^{p(k-1)}}{d^{p(k+1-j)}}|v_h|_{W^{j,p}(T)}^p\\
&\lss\frac{1}{d^{2p}}\|v_h\|_{L^p(T)}^p+\sum_{j=1}^{k}\frac{h^{p(k-j)}}{d^{p(k+1-j)}}|v_h|_{W^{1,p}(T)}^p\nonumber\\
&\lss\frac{1}{d^{2p}}\Bigl(\|v_h\|_{L^p(T)}^p+|v_h|_{W^{1,p}(T)}^p\Bigr).\nonumber
\end{align}
It follows from \eqref{eqn2:40} and \eqref{eqn2:42} that
\begin{align*}
\|D^2(\eta v_h-\tI(\eta v_h))\|_{L^p(T)}^p
&\lss h^{p(k-1)}|\eta v_h|_{W^{k+1,p}(T)}^p
\lss \frac{1}{d^{2p}}\Bigl(\|v_h\|_{L^p(T)}^p+|v_h|_{W^{1,p}(T)}^p\Bigr),\\
h^{-p}\|\grad(\eta v_h-\tI(\eta v_h))\|_{L^p(T)}^p
&\lss h^{p(k-1)}|\eta v_h|_{W^{k+1,p}(T)}^p
\lss \frac{1}{d^{2p}}\Bigl(\|v_h\|_{L^p(T)}^p+|v_h|_{W^{1,p}(T)}^p\Bigr),\\
h^{-2p}\|\eta v_h-\tI(\eta v_h)\|_{L^p(T)}^p
&\lss h^{p(k-1)}|\eta v_h|_{W^{k+1,p}(T)}^p
\lss \frac{1}{d^{2p}}\Bigl(\|v_h\|_{L^p(T)}^p+|v_h|_{W^{1,p}(T)}^p\Bigr).
\end{align*}
Using the previous three estimates and Lemma \ref{lem2:2} we get
\begin{align*}
&\|\eta v_h-\tI(\eta v_h)\|_{\Wh(D)}^p \lss \sum_{T\in\Th}\|D^2(\eta v_h-\tI(\eta v_h))\|_{L^p(T)}^p\\
&\qquad +\sum_{e\in\EI}h_e^{1-p}\|[\grad(\eta v_h-\tI(\eta v_h))]\|_{L^p(e)}^p
 +\sum_{e\in\EI}h_e^{1-2p}\|[\eta v_h-\tI(\eta v_h)]\|_{L^p(e)}^p\\
&\qquad +\sum_{e\in\EB}h_e^{1-2p}\|\eta v_h-\tI(\eta v_h)\|_{L^p(e)}^p\\
&\lss \sum_{\substack{T\in\Th \\ T\cap D\neq\emptyset}}\|D^2(\eta v_h-\tI(\eta v_h))\|_{L^p(T)}^p 
+\sum_{\substack{T\in\Th \\ T\cap D\neq\emptyset}}h^{-p}\|\grad(\eta v_h-\tI(\eta v_h))\|_{L^p(T)}^p \\
&\qquad+\sum_{\substack{T\in\Th \\ T\cap D\neq\emptyset}}h^{-2p}\|\eta v_h-\tI(\eta v_h)\|_{L^p(T)}^p \\
&\lss \sum_{\substack{T\in\Th \\ T\cap D\neq\emptyset}}\frac{1}{d^{2p}}\Bigl(\|v_h\|_{L^p(T)}^p+|v_h|_{W^{1,p}(T)}^p\Bigr)
\lss \frac{1}{d^{2p}}\Bigl(\|v_h\|_{L^p(T)}^p+\|\grad_h v_h\|_{L^p(T)}^p\Bigr).
\end{align*}
Thus, \eqref{eqn2:39} holds.

Finally, the proof of \eqref{eqn2:43} is similar to that of \eqref{eqn2:39} except one minor detail.  
Since $k\geq 2$, by \eqref{eqn2:41} and an inverse inequality we get
\begin{align}\label{eqn2:44}
&h^{p(k-1)}|\eta v_h|_{W^{k+1,p}(T)}^p\lss \sum_{j=0}^{k} \frac{h^{p(k-1)}}{d^{p(k+1-j)}}|v_h|_{W^{j,p}(T)}^p\\
&\qquad =h^p\Bigl(\sum_{j=0}^{k} \frac{h^{p(k-2)}}{d^{p(k+1-j)}}|v_h|_{W^{j,p}(T)}^p\Bigr)\nonumber\\
&\qquad =h^p\Bigl(\frac{1}{d^{3p}}\|v_h\|_{L^p(T)}^p+\frac{1}{d^{2p}}|v_h|_{W^{1,p}(T)}^p
+\sum_{j=2}^k\frac{h^{p(k-j)}}{d^{p(k+1-j)}}|v|_{W^{2,p}(T)}^p\Bigr)\nonumber\\
&\qquad \lss \frac{h^p}{d^{3p}}\|v_h\|_{W^{2,p}(T)}^p.\nonumber
\end{align}
Thus,  we can obtain (as in the derivation on \eqref{eqn2:39}) using Lemma \ref{lem2:55} that
\begin{align*}
\|\eta v_h-\tI(\eta v_h)\|_{\Wh(D)} 
&\lss \frac{h}{d^3}\sum_{\substack{T\in\Th \\ T\cap D\neq\emptyset}}\|v_h\|_{W^{2,p}(T)}\\
&= \frac{h}{d^3}\Bigl(\|v_h\|_{L^p(D_h)} + \|\grad_h v_h\|_{L^p(D_h)} + \|D_h^2 v_h\|_{L^p(D_h)} \Bigr)\\
&\lss  \frac{h}{d^3}\|v_h\|_{\Wh(D_h)}.
\end{align*}
The proof is complete.
\end{proof}

\end{document}